\documentclass{amsart}
\usepackage[T1]{fontenc}

\usepackage{amssymb}
\usepackage{amsmath}
\usepackage{amsthm}
\usepackage{pgfplots} 
\pgfplotsset{compat=1.9}
\usepackage{graphicx}
\usepackage[utf8]{inputenc}
\usepackage{tikz}
\usepackage{tikz-cd}
\usepackage{bbm}
\usepackage{subcaption}
\usepackage[boxruled]{algorithm2e}
\usepackage{mathtools}
\usepackage{lipsum}
\usepackage[title,titletoc]{appendix}
\usepackage{booktabs}
\usepackage{here}
\usepackage{natbib,enumerate} 
\usepackage{hyperref}
\usepackage{faktor}
\usepackage{comment}
\usepackage{bm}
\usepackage{adjustbox}
\usepackage{circuitikz}
\usepackage{moresize}
\usepackage{pdflscape}
\usepackage{mathpazo}
\usepackage{euler}
\usepackage{faktor}
\usepackage[a4paper,left=3cm,right=3cm,top=3cm,bottom=3cm]{geometry}
\usepackage{microtype}

\renewcommand{\phi}{\varphi}

\theoremstyle{definition}
\newtheorem{theorem}{Theorem}
\numberwithin{theorem}{section}

\newtheorem{lemma}[theorem]{Lemma}
\newtheorem{corollary}[theorem]{Corollary}
\newtheorem{remark}[theorem]{Remark}
\newtheorem{definition}[theorem]{Definition}
\newtheorem{example}[theorem]{Example}
\newtheorem{notation}[theorem]{Notation}

\newcommand{\catname}[1]{{\textbf{#1}}}
\newcommand{\Set}{\catname{Set}}

\newcommand{\C}{\catname{C}}

\newcommand{\A}{\catname{A}}

\newcommand{\op}{\text{op}}

\newcommand{\id}{\text{id}}

\newcommand{\G}{\catname{G}}

\DeclareMathOperator{\X}{\mathcal{X}}
\DeclareMathOperator{\Y}{\mathcal{Y}}
\DeclareMathOperator{\PP}{\mathcal{P}}

\RequirePackage{tikz-cd}
\RequirePackage{amssymb}
\usetikzlibrary{calc}
\usetikzlibrary{decorations.pathmorphing}

\tikzset{curve/.style={settings={#1},to path={(\tikztostart)
    .. controls ($(\tikztostart)!\pv{pos}!(\tikztotarget)!\pv{height}!270:(\tikztotarget)$)
    and ($(\tikztostart)!1-\pv{pos}!(\tikztotarget)!\pv{height}!270:(\tikztotarget)$)
    .. (\tikztotarget)\tikztonodes}},
    settings/.code={\tikzset{quiver/.cd,#1}
        \def\pv##1{\pgfkeysvalueof{/tikz/quiver/##1}}},
    quiver/.cd,pos/.initial=0.35,height/.initial=0}

\tikzset{tail reversed/.code={\pgfsetarrowsstart{tikzcd to}}}
\tikzset{2tail/.code={\pgfsetarrowsstart{Implies[reversed]}}}
\tikzset{2tail reversed/.code={\pgfsetarrowsstart{Implies}}}
\tikzset{no body/.style={/tikz/dash pattern=on 0 off 1mm}}

\usepackage{xcolor}
\hypersetup{
    colorlinks,
    linkcolor={red!50!black},
    citecolor={blue!50!black},
    urlcolor={blue!80!black}
}

\begin{document}
\title{The 2-Monoidal Structure of Symmetric Sequences}
\author[Patrick Dan and Claudio Pfammatter]{Patrick Dan and Claudio Pfammatter}
\address{Master's program in Mathematics, Laboratory for Topology and Neuroscience, EPFL, Switzerland}
\email{patrick-cristian.dan@epfl.ch \\
claudio.pfammatter@epfl.ch}
\begin{abstract}
We prove a conjecture proposed by Dwyer and Hess in \cite{Operads_2013_Hess}, namely that the category of symmetric sequences, when equipped with the composition and matrix monoidal structures, forms a 2-monoidal category in the sense of \cite{Aguiar2010}. Our approach leverages a detailed analysis of the interplay between these two monoidal structures, offering new insights into their behavior. As a consequence, we are able to show that every operad admits a Day-operad structure.
\end{abstract}
\maketitle

\tableofcontents

\section{Introduction}
In this paper, we confirm an idea proposed by Dwyer and Hess in \cite{Operads_2013_Hess} (see Remark 1.21). The conjecture asserts that the category of symmetric sequences, equipped with the composition and matrix monoidal structures, is a 2-monoidal category, as defined by Aguiar and Mahajan in \cite{Aguiar2010}. While Dwyer and Hess laid the theoretical groundwork, our work provides a complete proof of this result, expanding their ideas and filling in details of intermediate results.

Symmetric sequences play a crucial role in operad theory, especially when exploring algebraic structures with multiple inputs and outputs. The two monoidal structures of interest are foundational in this context. The composition, used to define operads, allows for structured combinations of objects, while the matrix product endows symmetric sequences with a closed monoidal structure. \\

\noindent
\textbf{Theorem 3.1.} The category $(\catname{sSeq}, \square, \mathcal{J}, \circ, \mathcal{J})$ is a $2$-monoidal category in the sense of \cite{Aguiar2010}. \\

Our proof includes intermediate results (some stated without proof in \cite{Operads_2013_Hess}) and leverages the equivalence of categories introduced by Dwyer and Hess to establish a solid foundation. We extend their theoretical framework by incorporating new categories beyond the original $\catname{Cospan}$ and $\catname{Comp}$, which are essential to our generalization. This refined approach not only confirms the $2$-monoidal structure but also offers new insights into the algebraic and categorical properties of symmetric sequences, opening further avenues for the study of operads with multiple interacting monoidal structures.

The paper is structured as follows. Section \ref{section_operads} reviews the necessary background on operads and symmetric sequences, focusing on the two aforementioned monoidal structures. Section \ref{section_main_theorem} presents the main theorem and its proof, while Section \ref{section_implications} discusses the implications of the 2-monoidal structure in other areas of operad theory. The most notable corollary of Theorem \ref{2-monoidal} is the following.\\

\noindent
\textbf{Corollary 4.6.} Every operad admits a natural Day-operad structure. \\

Throughout this article, we work in the context of concrete categories to facilitate understanding through element-chasing arguments, making certain proofs more intuitive. However, the structures and proofs presented here hold in a more abstract categorical setting, though the concrete approach provides clearer insight into the technical aspects of our constructions. \\

\noindent
\textbf{Acknowledgements.} The present research was carried out during a summer internship in the Laboratory for Topology and Neuroscience at EPFL. We would like to thank the EPFL Summer in the Lab and Student Support Programs for making this possible, as well as the Domaine de Villette Foundation for their support. We are especially grateful to Professor Kathryn Hess Bellwald and Professor Jérôme Scherer for their continuous encouragement and kind guidance, and we would also like to thank the PhD students in the lab for their invaluable help and support.

\section{Operads and their algebras}
\label{section_operads}
This section introduces the ubiquitous concept of operads: objects that model $n$-ary operations. Intuitively, one can think of these as trees with a central vertex of degree $n+1$, having $n$ incoming arrows and one outgoing arrow, together with a coherent way of gluing such trees together. For a comprehensive and detailed treatment, which we follow closely in this section, we refer to \cite{markl2002operads}.

From now on, unless otherwise specified, any monoidal category is assumed to be symmetric, closed, and cocomplete. Let $(\catname{C}, \otimes, \catname{1})$ be a monoidal category with those properties.

\subsection{\texorpdfstring{$\G$}{G}-sequences in monoidal categories}
Before introducing operads, we define the category of symmetric (or more generally $\G$-) sequences and its different monoidal structures. We begin by defining some basic tools.
\begin{definition}
\label{G}
Let $\{G_n \; | \; n \in \mathbb{N} \}$ be a set of groups. Denote by $\G$ the totally disconnected groupoid given by $\text{Ob}(\G) = \mathbb{N}$ and $\G(n,n) = G_n \in \text{Ob}(\textbf{Grp})$ for all $n \in \mathbb{N}$, where the groups $G_n$ are such that for all $m,n$ we have a group homomorphism $$\psi_{m,n}:G_m \times G_n \rightarrow G_{mn}$$ that is functorial in $m$ and $n$. 

We define a bifunctor $\nu_{\text{mult}}: \G \times \G \rightarrow \G$ given on objects by $\nu_{\text{mult}}(m,n) = mn$ and on morphisms by the homomorphisms $\psi_{-,-}$.
\end{definition}

Such a groupoid can be obtained by considering the symmetric groups $S_n$ and the homomorphisms $$S_n \times S_m \rightarrow S_{mn}$$ given by mapping a pair $(\sigma,\tau)$ to the permutation of $\{1,...,mn\}$ which separates the set into $n$ blocks of size $m$, applies $\tau$ to each block and finally applies $\sigma$ to the blocks themselves. We denote this groupoid by $\mathbf{\Sigma}$.

\begin{definition}
For a groupoid $\G$ as described in Definition \ref{G}, we define the category $\catname{GSeq}$ of $\G$-\textit{sequences} in $\C$ as the functor category $\text{Fun}(\G^\op,\C)$. In the case of the groupoid $\mathbf{\Sigma}$, we denote this category by $\catname{sSeq}$.
\end{definition}

Unravelling the definitions, a $\G$-sequence $\X$ in the monoidal category $\C$ is given by a sequence of objects $(\X(0),\X(1),\X(2),...)$ such that $\X(n)$ admits a right action of $G_n$ for all non-negative integers $n$. If $\C$ is a concrete category, the elements of $\X(n)$ are often called \textit{of arity $n$}, and $\X(n)$ itself is called the \textit{arity $n$-piece}. A morphism between two $\G$-sequences is given by a sequence of $G_n$-equivariant morphisms in $\C$. We visualize the situation in the following diagram.
\[\begin{tikzcd}
	{\X(0)} & {\X(1)} & {\X(2)} & {\X(3)} & \cdots \\
	{\Y(0)} & {\Y(1)} & {\Y(2)} & {\Y(3)} & \cdots
	\arrow["{f_0}", from=1-1, to=2-1]
	\arrow["{f_1}", from=1-2, to=2-2]
	\arrow["{f_2}", from=1-3, to=2-3]
	\arrow["{f_3}", from=1-4, to=2-4]
\end{tikzcd}\]

\subsection{The composition monoidal product}
We now introduce two important monoidal structures on $\catname{sSeq}$; the first is highly non-symmetric. While the second admits a right adjoint and is thus closed, the first does not (at least a priori). The second structure generalizes to any category $\catname{GSeq}$.

\begin{definition}
The \textit{composition monoidal structure} on $\catname{sSeq}$ is given on objects by $$\left(\X \circ \Y\right)(n) = \coprod_{\substack{k \\n_1+ \ldots +n_k=n}} \X(k) \otimes_{S_k} (\Y(n_1) \otimes \Y(n_2) \otimes \ldots \Y(n_k)) \otimes_{S_{n_1} \times \ldots \times S_{n_k}}\coprod_{S_n}\textbf{1}$$ where the action of $S_k$ on $\Y(n_1) \otimes \Y(n_2) \otimes \ldots \Y(n_k)$ permutes the different factors, and the action of $S_{n_1} \times \ldots \times S_{n_k}$ on $\coprod_{S_n}\textbf{1}$ is given by pulling back the action of $S_n$ that permutes the factors by the homomorphism $$S_{n_1} \times \ldots \times S_{n_k} \rightarrow S_n$$ which constructs a permutation by blocks from the starting ones. Note that $\catname{1}$ is the unit object in the monoidal category $\catname{C}$. On morphisms, it is given by $$(f \circ g)_n = \coprod_{\substack{k \\n_1+\ldots+n_k=n}} f_k \otimes_{S_k} (g_{n_1} \otimes g_{n_2} \otimes \ldots g_{n_k}) \otimes_{S_{n_1} \times \ldots \times S_{n_k}}\coprod_{S_n} \id_{\textbf{1}}.$$
\end{definition}

The associativity and left and right unit constraints for this monoidal structure are straightforward. The monoidal unit is given by the sequence $\mathcal{J} = (\varnothing, \textbf{1}, \varnothing, ...)$ with trivial $S_n$ actions, where $\varnothing$ denotes the initial object in $\catname{C}$.

\begin{definition} \label{definition_operad}
An \textit{operad} is a monoid in the monoidal category $(\catname{sSeq}, \circ, \mathcal{J})$. Explicitly, it consists of a symmetric sequence $\PP$ together with morphisms of symmetric sequences $\mu:\PP \circ \PP \rightarrow \PP$, called the \textit{multiplication}, and $\eta: \mathcal{J} \rightarrow \PP$, called the \textit{unit}, such that the following diagrams commute. 
\[\begin{tikzcd}
	{(\PP \circ \PP) \circ \PP} && {\PP \circ (\PP \circ \PP)} \\
	{\PP \circ \PP} && {\PP \circ \PP} \\
	& \PP
	\arrow["{{{a_{\PP,\PP,\PP}}}}", from=1-1, to=1-3]
	\arrow["{{\mu \circ \id_{\PP}}}"', from=1-1, to=2-1]
	\arrow["{{\id_{\PP} \circ \mu}}", from=1-3, to=2-3]
	\arrow["\mu"', from=2-1, to=3-2]
	\arrow["\mu", from=2-3, to=3-2]
\end{tikzcd}\]
\[\begin{tikzcd}
	{\mathcal{J} \circ \PP} && {\PP \circ \PP} && {\PP \circ \mathcal{J}} && {\PP \circ \PP} \\
	& \PP &&&& \PP
	\arrow["{{\eta \circ \id_{\PP}}}", from=1-1, to=1-3]
	\arrow["{{l_{\PP}}}"', from=1-1, to=2-2]
	\arrow["\mu", from=1-3, to=2-2]
	\arrow["{{\id_{\PP} \circ \eta}}", from=1-5, to=1-7]
	\arrow["{{r_{\PP}}}"', from=1-5, to=2-6]
	\arrow["\mu", from=1-7, to=2-6]
\end{tikzcd}\]
The arrows labeled $a, l$, and $r$ in the diagrams above denote the associator and the unitors in ($\catname{sSeq}, \circ, \mathcal{J}$).
\end{definition}

Viewing elements of $\PP(n)$ as described above, an operad structure on $\PP$ corresponds precisely to a gluing operation on trees. At the level of $n$-ary operations, the map $$\mu_n: \coprod_{\substack{k \\n_1+ \ldots +n_k=n}} \PP(k) \otimes_{S_k} (\PP(n_1) \otimes \PP(n_2) \otimes \ldots \PP(n_k)) \otimes_{S_{n_1} \times \ldots \times S_{n_k}}\coprod_{S_n}\textbf{1} \rightarrow \PP(n)$$ constructs a tree with $n$ incoming arrows from a tuple of trees with $n_1, \ldots, n_k$ incoming arrows, together with a separate $k$-ary tree that glues them. Visually, we have the following situation.

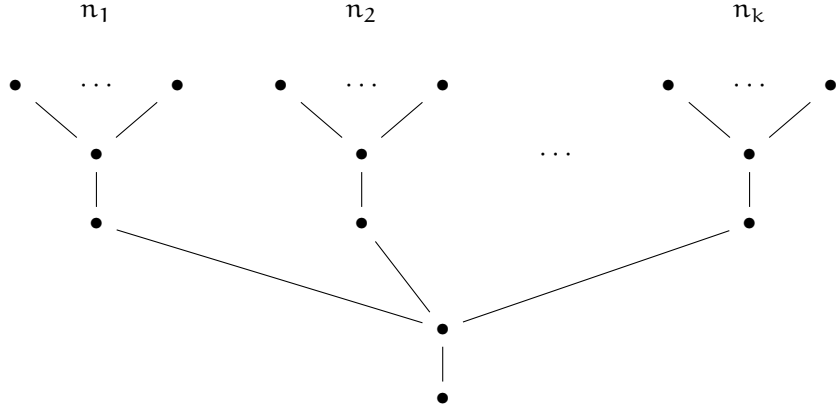
\[\begin{tikzcd}[column sep=small,row sep=scriptsize]
	& {n_1} &&&& {n_2} &&&&&& {n_k} \\
	\bullet & \cdots & \bullet && \bullet & \cdots & \bullet &&&& \bullet & \cdots & \bullet \\
	& \bullet &&&& \bullet &&& \cdots &&& \bullet \\
	& \bullet &&&& \bullet &&&&&& \bullet \\
	\\
	&&&&&& \bullet \\
	&&&&&& \bullet
	\arrow[no head, from=2-1, to=3-2]
	\arrow[no head, from=2-3, to=3-2]
	\arrow[no head, from=2-5, to=3-6]
	\arrow[no head, from=2-7, to=3-6]
	\arrow[no head, from=2-11, to=3-12]
	\arrow[no head, from=2-13, to=3-12]
	\arrow[no head, from=3-2, to=4-2]
	\arrow[no head, from=3-6, to=4-6]
	\arrow[no head, from=3-12, to=4-12]
	\arrow[no head, from=4-2, to=6-7]
	\arrow[no head, from=4-6, to=6-7]
	\arrow[no head, from=4-12, to=6-7]
	\arrow[no head, from=6-7, to=7-7]
\end{tikzcd}\]

\subsection{The matrix monoidal product}
We now turn to the definition of the second important monoidal structure on symmetric sequences, following the presentation of Dwyer and Hess in \cite{Operads_2013_Hess}. First, we introduce a more general construction of a closed monoidal structure on the functor category $\text{Fun}(\catname{C},\catname{Set})$.

\begin{definition}
Let $\C$ be a (locally small) monoidal category and $X,Y \in \text{Ob}(\text{Fun}(\C,\textbf{Set}))$. We define the \textit{Day convolution} of $X$ and $Y$ as the left Kan extension of their external product $$X\Tilde{\times}Y: (c,d) \mapsto X(c) \times Y(d)$$ along the bifunctor given by the internal monoidal law $\otimes$ of $\C$. 
\[\begin{tikzcd}
	{\C \times \C} && \Set \\
	& {\C}
	\arrow[""{name=0, anchor=center, inner sep=0}, "{{X \Tilde{\times} Y}}", from=1-1, to=1-3]
	\arrow["\otimes"', from=1-1, to=2-2]
	\arrow["{{\text{Lan}_\otimes (X\Tilde{\times} Y) =: X \otimes_{\mathrm{Day}}Y }}"', dashed, from=2-2, to=1-3]
	\arrow[shorten <=3pt, Rightarrow, from=0, to=2-2]
\end{tikzcd}\]

The Day convolution, denoted by $\otimes_{\mathrm{Day}}$, assembles into a bifunctor and gives $\text{Fun}(\C,\Set)$ the structure of a closed monoidal category, as shown in Proposition 3.1 of \cite{nlab:day_convolution}. 
\end{definition} 


\begin{remark}
A direct computation of the left Kan extension yields the following formula for the Day convolution of two functors in terms of a coend: $$X \otimes_{\mathrm{Day}}Y(c) = \int^{(c_1,c_2) \in \C \times \C} \C(c_1 \otimes c_2,c) \times X(c_1) \times Y(c_2).$$
\end{remark}

Applying Day convolution to $\catname{GSeq}$ with the monoidal structure $\nu_{\text{mult}}$ on $\G$ provides a closed monoidal law $(-) \square (-)$. Dwyer and Hess call this the \textit{matrix monoidal structure} in \cite{Operads_2013_Hess}. On objects of $\catname{GSeq}$, it is given explicitly by the formula $$\left(\X \square \Y\right)(n) = \coprod_{lm=n} (\X(l) \times \Y(m)) \times_{G_l \times G_m}G_n.$$ Note that the monoidal unit is $\mathcal{J}$, the same as for the composition structure.
\begin{definition}
    A \textit{Day-operad} is a monoid in the monoidal category $(\catname{sSeq}, \square, \mathcal{J})$.
\end{definition}

\section{The interplay between \texorpdfstring{$\square$}{square} and \texorpdfstring{$\circ$}{circ}}
\label{section_main_theorem}

We are interested in the interplay between the composition and matrix monoidal structures. In \cite{Operads_2013_Hess}, Dwyer and Hess prove the existence of a natural transformation 
\begin{equation} \label{natural_sigma} \sigma_{-,-,-,-}: \Big((-) \circ (-)\Big) \square \Big((-) \circ (-)\Big) \Rightarrow \Big((-) \square (-)\Big) \circ \Big((-) \square (-)\Big)(23) \end{equation} between functors with domain $\catname{GSeq}^{\times 4}$ and codomain $\catname{GSeq}$. The functor $(23)$ swaps the second and third components. We will often drop the subscripts when writing $\sigma$ if the context is unambiguous.

\subsection{Symmetric sequences and 2-monoidal categories} We strengthen this result by proving that the two monoidal structures are even more tightly related.

\begin{theorem}
\label{2-monoidal}
The category $\left(\catname{sSeq}, \square, \mathcal{J}, \circ, \mathcal{J}\right)$ is a $2$-monoidal category in the sense of  \cite{Aguiar2010}.
\end{theorem}

\begin{definition}[\cite{Aguiar2010}]
\label{definition_2-monoidal}
    A \textit{2-monoidal category} is a 5-tuple $(\C,\square,I,\circ,J)$ where $(\C,\square,I)$ and $(\C,\circ,J)$ are monoidal categories with units $I$ and $J$ respectively, along with a transformation (called the interchange law) $$\zeta_{A,B,C,D}: (A \circ B) \square (C \circ D) \rightarrow (A \square C) \circ (B \square D)$$ that is natural in $A,B,C$, and $D$, and three morphisms $$\Delta_I: I \rightarrow I \circ I, \quad \mu_J: J \square J \rightarrow J, \quad \iota_J = \epsilon_I: I \rightarrow J,$$ such that the axioms below are satisfied.

    \textbf{Associativity.} The following diagrams commute.
\begin{equation}\label{diagram_2-monoidal_associativity_1}
\begin{tikzcd}
	{\Big((A \circ B)\square (C \circ D)\Big) \square (E \circ F)} & {(A \circ B)\square \Big((C \circ D) \square (E \circ F)\Big)} \\
	{\Big((A \square C)\circ (B \square D)\Big) \square (E \circ F)} & {(A \circ B) \square \Big((C \square E) \circ (D \square F)\Big)} \\
	{\Big((A \square C) \square E\Big) \circ \Big((B \square D) \square F\Big)} & {\Big(A \square (C \square E)\Big) \circ \Big(B \square (D \square F)\Big)}
	\arrow["{\scriptscriptstyle a^\square}", from=1-1, to=1-2]
	\arrow["{\scriptscriptstyle \zeta \square (\id \circ \id)}"', from=1-1, to=2-1]
	\arrow["{\scriptscriptstyle (\id \circ \id) \square \zeta}", from=1-2, to=2-2]
	\arrow["{\scriptscriptstyle \zeta}"', from=2-1, to=3-1]
	\arrow["{\scriptscriptstyle \zeta}", from=2-2, to=3-2]
	\arrow["{\scriptscriptstyle a^\square \circ a^\square}"', from=3-1, to=3-2]
\end{tikzcd}\end{equation}

\begin{equation}\label{diagram_2-monoidal_associativity_2}
\begin{tikzcd}
	{\Big((A \circ B)\circ C\Big) \square \Big((D \circ E) \circ F\Big)} & {\Big(A \circ (B\circ C)\Big) \square \Big(D \circ (E \circ F)\Big)} \\
	{\Big((A \circ B) \square (D \circ E)\Big) \circ (C \square F)} & {(A \square D) \circ \Big((B \circ C) \square (E \circ F)\Big)} \\
	{\Big((A \square D) \circ (B \square E)\Big) \circ (C \square F)} & {(A \square D) \circ \Big((B \square E) \circ (C \square F)\Big)}
	\arrow["{\scriptscriptstyle a^\circ \square a^\circ}", from=1-1, to=1-2]
	\arrow["{\scriptscriptstyle \zeta}"', from=1-1, to=2-1]
	\arrow["{\scriptscriptstyle \zeta}", from=1-2, to=2-2]
	\arrow["{\scriptscriptstyle \zeta \circ (\id \square \id)}"', from=2-1, to=3-1]
	\arrow["{\scriptscriptstyle (\id \square \id) \circ \zeta}", from=2-2, to=3-2]
	\arrow["{\scriptscriptstyle a^\circ}"', from=3-1, to=3-2]
\end{tikzcd}\end{equation}

\textbf{Unitality.} The following diagrams commute.
\begin{equation}\label{diagram_2-monoidal_unitality_1}
\begin{tikzcd}
	{I \square (A \circ B)} & {(I \circ I)\square(A \circ B)} && {(A \circ B) \square I} & {(A \circ B)\square (I \circ I)} \\
	{A \circ B} & {(I \square A) \circ (I \square B)} && {A \circ B} & {(A \square I) \circ (B \square I)}
	\arrow["{\scriptscriptstyle \Delta_I \square (\id \circ \id)}", from=1-1, to=1-2]
	\arrow["{\scriptscriptstyle \zeta}", from=1-2, to=2-2]
	\arrow["{\scriptscriptstyle (\id \circ \id) \square \Delta_I}", from=1-4, to=1-5]
	\arrow["{\scriptscriptstyle \zeta}", from=1-5, to=2-5]
	\arrow["{\scriptscriptstyle l^\square_{A \circ B}}", from=2-1, to=1-1]
	\arrow["{\scriptscriptstyle l^\square_A \circ l^\square_B}"', from=2-1, to=2-2]
	\arrow["{\scriptscriptstyle r^\square_{A \circ B}}", from=2-4, to=1-4]
	\arrow["{\scriptscriptstyle r^\square_A \circ r^\square_B}"', from=2-4, to=2-5]
\end{tikzcd}\end{equation}

\begin{equation}\label{diagram_2-monoidal_unitality_2}
\begin{tikzcd}
	{J \circ (A \square B)} & {(J \square J) \circ (A \square B)} && {(A \square B)\circ J} & {(A \square B) \circ (J \square J)} \\
	{A \square B} & {(J \circ A) \square (J \circ B)} && {A \square B} & { (A \circ J) \square (B \circ J)}
	\arrow["{\scriptscriptstyle \mu_J \circ (\id \square \id)}"', from=1-2, to=1-1]
	\arrow["{\scriptscriptstyle (\id \square \id) \circ \mu_J}"', from=1-5, to=1-4]
	\arrow["{\scriptscriptstyle l^\circ_{A \square B}}", from=2-1, to=1-1]
	\arrow["{\scriptscriptstyle l^\circ_A \square l^\circ_B}"', from=2-1, to=2-2]
	\arrow["{\scriptscriptstyle \zeta}"', from=2-2, to=1-2]
	\arrow["{\scriptscriptstyle r^\circ_{A \square B}}", from=2-4, to=1-4]
	\arrow["{\scriptscriptstyle r^\circ_A \square r^\circ_B}"', from=2-4, to=2-5]
	\arrow["{\scriptscriptstyle \zeta}"', from=2-5, to=1-5]
\end{tikzcd}\end{equation}

\textbf{Compatibility of units.} The units $I$ and $J$ are compatible in the following sense. $$(J, \mu_J, \iota_J) \; \text{is a monoid in} \; (\C, \square, I).$$ $$(I, \Delta_I, \epsilon_I) \; \text{is a comonoid in} \; (\C, \circ, J).$$ The arrows labeled $a$, $l$, and $r$ in the diagrams above denote the associativity and unit constraints of the respective monoidal categories.
\end{definition}

\subsection{Technical lemmas and notation}
Before proving the main result, we explain the construction of the above-mentioned natural transformation $$\sigma_{-,-,-,-}: \Big((-) \circ (-)\Big) \square \Big((-) \circ (-)\Big) \Rightarrow \Big((-) \square (-)\Big) \circ \Big((-) \square (-)\Big)(23) ,$$ by introducing relevant notation and technical lemmas. Some of these appear in \cite{Operads_2013_Hess} and are labeled accordingly.

\begin{definition}
\label{definition_A}
Let $\catname{A}$ be the groupoid in which the objects are maps between finite sets and the morphisms between two objects $f:T \rightarrow S$ and $g:T' \rightarrow S'$ are pairs of bijections $$\left(T \xrightarrow{k} T', S \xrightarrow{l} S'\right)$$ that make the following diagram commute.
\[\begin{tikzcd}
	T & S \\
	{T'} & {S'}
	\arrow["f", from=1-1, to=1-2]
	\arrow["k"', from=1-1, to=2-1]
	\arrow["l", from=1-2, to=2-2]
	\arrow["g"', from=2-1, to=2-2]
\end{tikzcd}\]

\end{definition}

The next lemma comes from \cite{Operads_2013_Hess}, but we present an alternative proof.

\begin{lemma} [\cite{Operads_2013_Hess}]
\label{equivalence}
For any symmetric closed and cocomplete monoidal category $\catname{C}$, the category $\catname{sSeq}  = \text{Fun}(\mathbf{\Sigma}^{\op}, \catname{C})$ is equivalent to the category $[\catname{A}^{\op},\catname{C}]_{\text{mult}}$ of functors $\catname{A}^{\op} \rightarrow \catname{C}$ that are multiplicative, i.e., monoidal with respect to $\coprod$ on $\catname{A}^{\op}$ and $\otimes$ on $\catname{C}$.
\end{lemma}
\begin{proof}
We will work in the context of concrete cartesian monoidal categories, and more specifically for $\catname{C} = \catname{Set}$. The arguments are analogous in the general case. 

Note that $\mathbf{\Sigma}^\op$ is a subcategory of $\catname{A}^\op$ in the obvious way via the maps $$\{1, \ldots, n\} \rightarrow \{1\}$$ of finite sets.

Define functors $$\catname{sSeq} \rightarrow [\catname{A}^{\op},\catname{C}]_{\text{mult}}: \mathcal{X} \mapsto \Phi_{\mathcal{X}}$$ and $$[\catname{A}^{\op},\catname{C}]_{\text{mult}} \rightarrow \catname{sSeq}: \Phi \mapsto \mathcal{X}_{\Phi}$$ as follows. For all $\Phi \in [\catname{A}^{\op}, \catname{C}]_{\text{mult}}$, the symmetric sequence $\mathcal{X}_{\Phi}$ is given by restricting $\Phi$ to the subcategory $\mathbf{\Sigma}^{\op}$ of $\catname{A}^\op$. For any symmetric sequence $\mathcal{X}$, define $\Phi_{\mathcal{X}}$ by $$\Phi_{\mathcal{X}}\left(T \xrightarrow{f} S \right) = \prod_{s \in S} \mathcal{X}\left(|f^{-1}(s)|\right).$$ 
By the universal property of the categorical product, we have a natural isomorphism $$\Phi_\mathcal{X}(- \coprod -) \xrightarrow{\cong} \Phi_\mathcal{X}(-) \times \Phi_\mathcal{X}(-)$$ making $\Phi_\mathcal{X}$ a multiplicative functor.

We now verify that these functors provide an equivalence of categories. Since the composition $\catname{sSeq} \rightarrow [\catname{A}^\op,\C]_{\text{mult}} \rightarrow \catname{sSeq}$ is the identity, $\Phi \mapsto \mathcal{X}_\Phi$ is surjective and thus essentially surjective. It remains to check bijectivity of the map \begin{equation} \label{full_faithful} [\catname{A}^\op, \C]_\text{mult}(\Phi, \Psi) \rightarrow \catname{sSeq}(\mathcal{X}_\Phi, \mathcal{X}_\Psi). \end{equation} Surjectivity follows by the same argument as above. A morphism between multiplicative functors $\Phi,\Psi: \catname{A}^\op \rightarrow \C$ is a monoidal natural transformation, i.e., a natural transformation $\tau: \Phi \Rightarrow \Psi$ such that
\[\begin{tikzcd}
	{\Phi(f \coprod g)} & {\Phi(f) \times \Phi(g)} \\
	{\Psi(f \coprod g)} & {\Psi(f) \times \Psi(g)}
	\arrow["\cong", from=1-1, to=1-2]
	\arrow["{\tau_{f \coprod g}}"', from=1-1, to=2-1]
	\arrow["{\tau_f \times \tau_g}", from=1-2, to=2-2]
	\arrow["\cong"', from=2-1, to=2-2]
\end{tikzcd}\]
commutes for all $f,g \in \catname{A}^\op$. We need to show that the map (\ref{full_faithful}) is injective. To do this, we show that by restricting a monoidal natural transformation to the subcategory $\mathbf{\Sigma}^\op$ of $\catname{A}^\op$, we do not lose any information. 

Let $f: T \rightarrow S$ be an object of $\catname{A}^\op$, which we may write as $$f = \coprod_{s \in S} \left(f^{-1}(\{s\}) \rightarrow \{s\} \right).$$ By multiplicativity of $\Phi$, we have $$\Phi \left( T \xrightarrow{f} S \right) = \prod_{s \in S} \Phi \left(f^{-1}(\{s\}) \rightarrow \{s\} \right),$$ which shows that $\Phi$ is determined by its values on set maps isomorphic to $\{1, \ldots,n\} \rightarrow \{1\}$. Moreover, the commutative diagram
\[\begin{tikzcd}
	{\Phi\left(\coprod_{s \in S}f^{-1}(\{s\}) \rightarrow \{s\}\right)} & {\prod_{s \in S} \Phi\left(f^{-1}(\{s\}) \rightarrow \{s\}\right)} \\
	{\Psi\left(\coprod_{s \in S}f^{-1}(\{s\}) \rightarrow \{s\}\right)} & {\prod_{s \in S} \Psi\left(f^{-1}(\{s\}) \rightarrow \{s\}\right)}
	\arrow["\cong", from=1-1, to=1-2]
	\arrow["{\tau_{\coprod_{s \in S}f^{-1}(\{s\}) \rightarrow \{s\}}}"', from=1-1, to=2-1]
	\arrow["{\prod_{s \in S} \tau_{f^{-1}(\{s\}) \rightarrow \{s\}}}", from=1-2, to=2-2]
	\arrow["\cong"', from=2-1, to=2-2]
\end{tikzcd}\]
shows that $\tau: \Phi \Rightarrow \Psi$ is also entirely determined by its values on set maps isomorphic to $\{1, \ldots,n\} \rightarrow \{1\}$. This proves that $$[\catname{A}^\op, \C]_\text{mult}(\Phi, \Psi) \rightarrow \catname{sSeq}(\mathcal{X}_\Phi, \mathcal{X}_\Psi)$$ is injective as well, and thus the categories $\catname{sSeq}$ and $[\catname{A}^\op,\C]_\text{mult}$ are equivalent.
\end{proof}

Recall that $(\catname{C},\otimes,\catname{1})$ is a symmetric, closed, and cocomplete monoidal category.

\begin{notation}
Let $\catname{Cospan}$ be the subcategory of $\catname{A}^{\op} \times \catname{A}^{\op}$ whose objects are pairs of maps with coinciding codomain. \\
Let $\catname{Comp}$ be the subcategory of $\catname{A}^{\op} \times \catname{A}^{\op}$ whose objects are composable pairs of maps. Let $\catname{3-Comp}$ be the subcategory of $(\catname{A}^{\op})^ {\times 3}$ whose objects are triples composable maps. 

For functors $\Phi, \Psi \in [\catname{A}^{\op}, \catname{C}]_{\text{mult}}$, we denote by $\Phi \times \Psi$ the external product of the two functors, defined by $$\left(f,f'\right) \mapsto \Phi\left(f\right) \otimes \Psi\left(f'\right),$$ where the tensor product takes place in $\catname{C}$. By abuse of notation, we use the same symbol for its restriction to either $\catname{Cospan}$ or $\catname{Comp}$.
\end{notation}

The next result appears without proof in \cite{Operads_2013_Hess}. We prove the first part; the second follows by a similar argument.
\begin{lemma} [\cite{Operads_2013_Hess}] \label{operations_kan}
Under the equivalence of Lemma \ref{equivalence},
\begin{enumerate}[1)]
    \item the left Kan extension of $\Phi \times \Psi$ along the fibered product functor 
    \begin{align*}
    \catname{Cospan} &\rightarrow \catname{A}^{\op} \\
    \left(T \xrightarrow{f} S, T' \xrightarrow{f'} S\right) &\mapsto \left(T \times_{S}T' \xrightarrow{f\times_{S} f'} S\right)
    \end{align*}
    corresponds to the matrix monoidal product of symmetric sequences,
    \item the left Kan extension of $\Phi \times \Psi$ along the composition functor \begin{align*}
    \catname{Comp} &\rightarrow \catname{A}^{\op} \\ \left(S\xrightarrow{f} R, T \xrightarrow{g} S\right) &\mapsto \left(T \xrightarrow{gf} R\right)
    \end{align*}
    corresponds to the composition product of symmetric sequences, and
    \item the unit for the external product of functors is $\Gamma: \catname{A}^{\op} \rightarrow \C$ specified by $\Gamma(f) = \textbf{1}$ if $f$ is a bijection and $\Gamma(f) = \varnothing$ otherwise.
\end{enumerate}
\end{lemma}

\begin{proof}
    We prove the first part of the lemma in the case $\C = \Set$.
    
    Let us denote by $\gamma: \catname{Cospan} \rightarrow \catname{A}^\op$ the fibered product functor. We need to show that the left Kan extension
\[\begin{tikzcd}
	{\catname{Cospan}} && \Set \\
	& {\A^\op}
	\arrow[""{name=0, anchor=center, inner sep=0}, "{\Phi \times \Psi}", from=1-1, to=1-3]
	\arrow["\gamma"', from=1-1, to=2-2]
	\arrow["{\text{Lan}_\gamma (\Phi \times \Psi)}"', dashed, from=2-2, to=1-3]
	\arrow[shorten <=3pt, Rightarrow, from=0, to=2-2]
\end{tikzcd}\]    
    of $\Phi \times \Psi$ along $\gamma$ corresponds to the matrix monoidal product under the equivalence of the categories $[\A^\op,\Set]_\text{mult}$ and $\catname{sSeq}$. 

    In the following computation, we use the coend formula of \cite{nlab:kan_extension} for the left Kan extension to get
    \begin{align*}
        \text{Lan}_\gamma (\Phi \times \Psi) \left( \{1, \ldots, n\} \rightarrow \{1\} \right) &= \int^{(f,f') \in \catname{Cospan}} \coprod_{\A^\op\left( f \times_S f',  \{1, \ldots, n\} \rightarrow \{1\}\right)} \Phi(f) \times \Psi(f') \\
        &= \int^{(f,f') \in \catname{Cospan} \; \text{satisfying} \; (\star)} \coprod_{S_n} \Phi(f) \times \Psi(f').
    \end{align*}
    The coproduct appearing in the first line of the above computation is indexed by \begin{equation} \label{set_A} \A^\op\left( f \times_S f',  \{1, \ldots, n\} \rightarrow \{1\}\right)\end{equation} where $(f,f') \in \catname{Cospan}$, hence the maps are of the form $f: T \rightarrow S$ and $f': T' \rightarrow S$. Note that by the definition of morphisms in the category $\A^\op$, the set (\ref{set_A}) is either empty or can be identified with $S_n$. The latter applies if and only if $|S| = 1$ and $|T||T'| = n$. We denote this condition by $(\star)$. The above-mentioned identification is the reason why some of the coproducts in this proof are indexed by $S_n$.
    
    Recall that our goal is to show that $\text{Lan}_\gamma (\Phi \times \Psi)$ corresponds to $\mathcal{X}_\Phi \square \mathcal{X}_\Psi$ under the equivalence of \ref{equivalence}. Since $$\left(\mathcal{X}_\Phi \square \mathcal{X}_\Psi\right) (n) = \coprod_{lm = n} \left( \Phi(\{1, \ldots, l\} \rightarrow \{1\}) \times \Psi(\{1, \ldots, m\} \rightarrow \{1\}) \right) \times_{S_l \times S_m} S_n$$ it suffices to show that $$\coprod_{lm = n} \left( \Phi(\{1, \ldots, l\} \rightarrow \{1\}) \times \Psi(\{1, \ldots, m\} \rightarrow \{1\}) \right) \times_{S_l \times S_m} S_n$$ satisfies the universal property of $$\int^{(f,f') \in \catname{Cospan} \; \text{satisfying} \; (\star)} \coprod_{S_n} \Phi(f) \times \Psi(f')$$ to conclude.

    The coend above has the universal property of the pushout of the following diagram
\[\begin{tikzcd}
	{\coprod_{\A^\op \left( f \times_S f', \{1, \ldots, n\} \rightarrow \{1\} \right)} \Phi(g) \times \Psi(g')} & {\coprod_{\A^\op \left( g \times_R g', \{1, \ldots, n\} \rightarrow \{1\} \right)} \Phi(g) \times \Psi(g')} \\
	{\coprod_{\A^\op \left( f \times_S f', \{1, \ldots, n\} \rightarrow \{1\} \right)} \Phi(f) \times \Psi(f')}
	\arrow[from=1-1, to=1-2]
	\arrow[from=1-1, to=2-1]
\end{tikzcd}\]
    where $(g,g') \rightarrow (f,f')$ is a morphism in $\catname{Cospan}$. If $(f,f')$ doesn't satisfy $(\star)$, the diagram is empty. On the other hand if $(f,f')$ satisfies $(\star)$, then $(g,g')$ does as well because both pairs are connected via a morphism in $\A^\op \times \A^\op$ and thus we may write $l_f=|\text{dom} f|$ and $m_{f'}=|\text{dom} f'|$, as well as $l_g=|\text{dom} g|$ and $m_{g'}=|\text{dom} g'|$, and observe that the following diagram
\[\begin{tikzcd}[column sep=small]
	{\coprod_{S_n} \Phi(l_g) \times \Psi(m_{g'})} & {\coprod_{S_n} \Phi(l_g) \times \Psi(m_{g'})} \\
	{\coprod_{S_n} \Phi(l_f) \times \Psi(m_{f'})} & {\coprod_{lm = n} \left( \Phi(l) \times \Psi(m) \right) \times_{S_l \times S_m} S_n}
	\arrow[from=1-1, to=1-2]
	\arrow[from=1-1, to=2-1]
	\arrow[from=1-2, to=2-2]
	\arrow[from=2-1, to=2-2]
\end{tikzcd}\]
    commutes since $l_f = l_g$ and $m_{f'} = m_{g'}$, where $\Phi(l)$ denotes $\Phi(\{1, \ldots, l\} \rightarrow \{1\})$. 
    
    It remains to show that $$\coprod_{lm = n} \left( \Phi(\{1, \ldots, l\} \rightarrow \{1\}) \times \Psi(\{1, \ldots, m\} \rightarrow \{1\}) \right) \times_{S_l \times S_m} S_n$$ is indeed the universal cowedge. For this, given a set $Z$ and a family of commutative diagrams of the form
\[\begin{tikzcd}[column sep=small]
	{\coprod_{S_n} \Phi(l_g) \times \Psi(m_{g'})} & {\coprod_{S_n} \Phi(l_g) \times \Psi(m_{g'})} \\
	{\coprod_{S_n} \Phi(l_f) \times \Psi(m_{f'})} & Z
	\arrow[from=1-1, to=1-2]
	\arrow[from=1-1, to=2-1]
	\arrow[from=1-2, to=2-2]
	\arrow[from=2-1, to=2-2]
\end{tikzcd}\]
    where $(g,g') \rightarrow (f,f')$ is a morphism in $\catname{Cospan}$ satisfying $(\star)$, we observe that there are unique morphisms $$\left( \Phi(l) \times \Psi(m) \right) \times_{S_l \times S_m} S_n \rightarrow Z$$ making the diagrams
\[\begin{tikzcd}[column sep=small]
	{\coprod_{S_n} \Phi(l_g) \times \Psi(m_{g'})} & {\coprod_{S_n} \Phi(l_g) \times \Psi(m_{g'})} \\
	{\coprod_{S_n} \Phi(l_f) \times \Psi(m_{f'})} & {\left( \Phi(l) \times \Psi(m) \right) \times_{S_l \times S_m} S_n} \\
	&& Z
	\arrow[from=1-1, to=1-2]
	\arrow[from=1-1, to=2-1]
	\arrow[from=1-2, to=2-2]
	\arrow[curve={height=-30pt}, from=1-2, to=3-3]
	\arrow[from=2-1, to=2-2]
	\arrow[curve={height=30pt}, from=2-1, to=3-3]
	\arrow["{\exists !}"', dashed, from=2-2, to=3-3]
\end{tikzcd}\]
    commute. By taking all morphisms $(g,g') \rightarrow (f,f')$ in $\catname{Cospan}$ satisfying $(\star)$ into account, we get a commutative diagram of the form
\[\begin{tikzcd}[column sep=small]
	{\coprod_{S_n} \Phi(l_g) \times \Psi(m_{g'})} & {\coprod_{S_n} \Phi(l_g) \times \Psi(m_{g'})} \\
	{\coprod_{S_n} \Phi(l_f) \times \Psi(m_{f'})} & {\coprod_{lm = n} \left( \Phi(l) \times \Psi(m) \right) \times_{S_l \times S_m} S_n} \\
	&& Z
	\arrow[from=1-1, to=1-2]
	\arrow[from=1-1, to=2-1]
	\arrow[from=1-2, to=2-2]
	\arrow[curve={height=-30pt}, from=1-2, to=3-3]
	\arrow[from=2-1, to=2-2]
	\arrow[curve={height=30pt}, from=2-1, to=3-3]
	\arrow["{\exists !}"', dashed, from=2-2, to=3-3]
\end{tikzcd}\]
    by the universal property of the coproduct. 
    
    This proves the first part. The second part follows by a similar argument, and the third is a direct verification. 
\end{proof}

\begin{notation} \label{notation}
The following categories will be important in the construction of $\sigma$. Their objects can be visualized using diagrams.

\begin{itemize}

\item Let $\catname{Cospan(Comp)}$ denote the subcategory of $\catname{Comp} \times \catname{Comp}$ the objects of which are pairs of pairs $(f,g), \left(f',g'\right)$ such that $\left(f,f'\right) \in \catname{Cospan}$.
\[\begin{tikzcd}
	T \\
	S \\
	R & {S'} & {T'}
	\arrow["g"', from=1-1, to=2-1]
	\arrow["f"', from=2-1, to=3-1]
	\arrow["{f'}", from=3-2, to=3-1]
	\arrow["{g'}", from=3-3, to=3-2]
\end{tikzcd}\]

\item Let $\catname{Comp(Cospan)}$ denote the subcategory of $\catname{Cospan} \times \catname{Cospan}$ the objects of which are pairs of pairs $\Big(\left(f,f'\right),\left(g,g'\right)\Big)$ such that $$\text{Cod}(g) = \text{Cod}\left(g'\right) = \text{Dom}(f) \times_{\text{Cod}(f)} \text{Dom}\left(f'\right).$$
\begin{equation}
\label{diagram_diagonal}
\begin{tikzcd}
	& T & {T \times_{S \times_{R}S'}T'} \\
	S & {S \times_{R}S'} & {T'} \\
	R & {S'}
	\arrow["g"', from=1-2, to=2-2]
	\arrow[dashed, from=1-3, to=1-2]
	\arrow[color={rgb,255:red,255;green,0;blue,0}, dotted, from=1-3, to=2-2]
	\arrow[dashed, from=1-3, to=2-3]
	\arrow["f"', from=2-1, to=3-1]
	\arrow[dashed, from=2-2, to=2-1]
	\arrow[color={rgb,255:red,255;green,0;blue,0}, dotted, from=2-2, to=3-1]
	\arrow[dashed, from=2-2, to=3-2]
	\arrow["{g'}", from=2-3, to=2-2]
	\arrow["{f'}", from=3-2, to=3-1]
\end{tikzcd}\end{equation}

\end{itemize}
This terminology is coherent because there is a natural way of "composing" the two pairs by composing the two diagonal maps in the pullbacks.

The following categories will be important when applying $\sigma$ to a composition of six symmetric sequences by grouping four of them into two pairs. We provide visualizations of their objects to aid intuition.

\begin{itemize}

\item Let $\catname{Cospan(3-Comp)}$ denote the subcategory of $(\catname{3-Comp})^{\times 2}$ the objects of which are pairs $$\Big((f,g, h),\left(f',g',h'\right)\Big)$$ such that $\left(f,f'\right) \in \catname{Cospan}$.
\[\begin{tikzcd}
	Q \\
	T \\
	S \\
	R & {S'} & {T'} & {Q'}
	\arrow["h"', from=1-1, to=2-1]
	\arrow["g"', from=2-1, to=3-1]
	\arrow["f"', from=3-1, to=4-1]
	\arrow["{f'}"', from=4-2, to=4-1]
	\arrow["{g'}"', from=4-3, to=4-2]
	\arrow["{h'}"', from=4-4, to=4-3]
\end{tikzcd}\]

\item Let $\catname{Comp\Big(Cospan, Cospan(Comp)\Big)}$ denote the subcategory of the product category $\catname{Cospan} \times \catname{Cospan(Comp)}$ the objects of which are of the form $$\Bigg(\left(f,f'\right),\Big((g,h),\left(g',h'\right)\Big)\Bigg)$$ such that $\Big(\left(f,f'\right),\left(gh,g'h'\right) \Big) \in \catname{Comp(Cospan)}$. 
\[\begin{tikzcd}
	& Q && {Q \times_{S \times_{R}S'}Q'} \\
	& T \\
	S & {S \times_{R}S'} & {T'} & {Q'} \\
	R & {S'}
	\arrow["h"', from=1-2, to=2-2]
	\arrow[dashed, from=1-4, to=1-2]
	\arrow[color=red, dotted, from=1-4, to=3-2]
	\arrow[dashed, from=1-4, to=3-4]
	\arrow["g"', from=2-2, to=3-2]
	\arrow["f"', from=3-1, to=4-1]
	\arrow[dashed, from=3-2, to=3-1]
	\arrow[color=red, dotted, from=3-2, to=4-1]
	\arrow[dashed, from=3-2, to=4-2]
	\arrow["{{g'}}", from=3-3, to=3-2]
	\arrow["{{h'}}", from=3-4, to=3-3]
	\arrow["{{f'}}", from=4-2, to=4-1]
\end{tikzcd}\]

\item Let $\catname{3-Comp(Cospan)}$ denote the subcategory of $\catname{Cospan}^{\times 3}$ the objects of which are of the form $$\Big(\left(f,f'\right),\left(g,g'\right),\left(h,h'\right)\Big)$$ with $\Big(\left(f,f'\right),\left(g,g'\right)\Big) \in \catname{Comp(Cospan)}$ and $\Big(\left(g,g'\right),\left(h,h'\right)\Big) \in \catname{Comp(Cospan)}$.
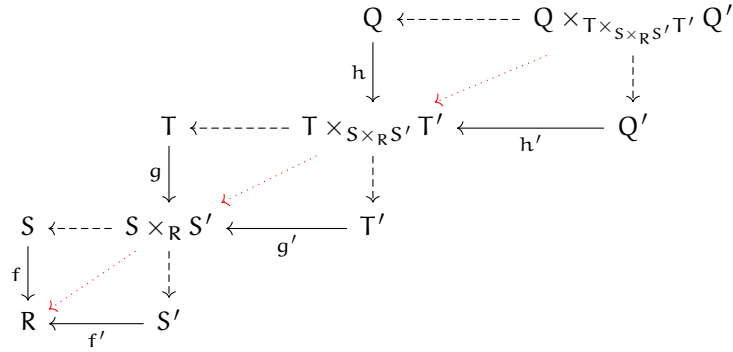
\[\begin{tikzcd}
	&& Q & {Q \times_{T \times_{S \times_{R}S'}T'} Q'} \\
	& T & {T \times_{S \times_{R}S'}T'} & {Q'} \\
	S & {S \times_{R}S'} & {T'} \\
	R & {S'}
	\arrow["h"', from=1-3, to=2-3]
	\arrow[dashed, from=1-4, to=1-3]
	\arrow[color=red, dotted, from=1-4, to=2-3]
	\arrow[dashed, from=1-4, to=2-4]
	\arrow["g"', from=2-2, to=3-2]
	\arrow[dashed, from=2-3, to=2-2]
	\arrow[color=red, dotted, from=2-3, to=3-2]
	\arrow[dashed, from=2-3, to=3-3]
	\arrow["{h'}", from=2-4, to=2-3]
	\arrow["f"', from=3-1, to=4-1]
	\arrow[dashed, from=3-2, to=3-1]
	\arrow[color=red, dotted, from=3-2, to=4-1]
	\arrow[dashed, from=3-2, to=4-2]
	\arrow["{g'}", from=3-3, to=3-2]
	\arrow["{f'}", from=4-2, to=4-1]
\end{tikzcd}\]
The red arrows show the natural way of "composing" these $3$ pairs, motivating the name of the category.
\item Let $\catname{Comp\Big(Cospan(Comp),Cospan\Big)}$ denote the subcategory of $$\catname{Cospan(Comp)} \times \catname{Cospan},$$ which is defined analogously to the category $$\catname{Comp\Big(Cospan, Cospan(Comp)\Big)},$$ but with the roles of the two pairs being reversed. 

\item Let $\catname{3-Cospan(Comp)}$ denote the subcategory of $\catname{Comp}^{\times 3}$ the objects of which are triplets of pairs $$\Big(\left(f,g\right),\left(f',g'\right),\left(f'',g''\right)\Big)$$ such that $\left(f,f'\right),\left(f',f''\right) \in \catname{Cospan}$. 
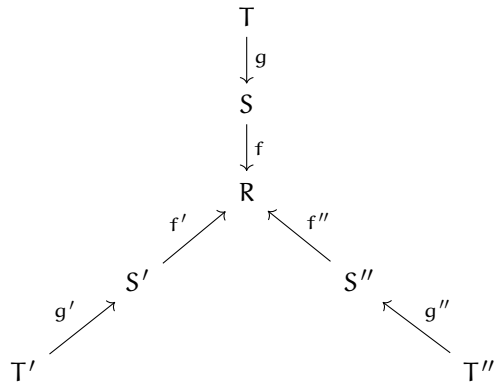
\[\begin{tikzcd}
	&& T \\
	&& S \\
	&& R \\
	& {S'} && {S''} \\
	{T'} &&&& {T''}
	\arrow["g", from=1-3, to=2-3]
	\arrow["f", from=2-3, to=3-3]
	\arrow["{f'}", from=4-2, to=3-3]
	\arrow["{f''}"', from=4-4, to=3-3]
	\arrow["{g'}", from=5-1, to=4-2]
	\arrow["{g''}"', from=5-5, to=4-4]
\end{tikzcd}\]

\item Let $\catname{Cospan\Big(Comp, Comp(Cospan)\Big)}$ denote the subcategory of the product category $\catname{Comp} \times \catname{Comp(Cospan)}$ the objects of which are of the form $$\Bigg(\left(f,g\right),\Big(\left(f',f''\right),\left(g',g''\right)\Big)\Bigg)$$ such that $\left(f,f'\right),\left(f',f''\right) \in \catname{Cospan}$. 
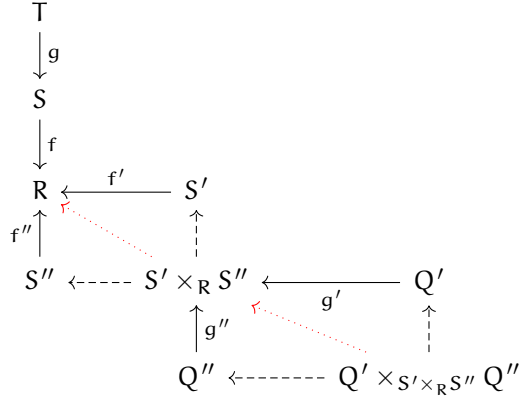
\[\begin{tikzcd}
	T \\
	S \\
	R & {S'} \\
	{S''} & {S' \times_{R}S''} & {Q'} \\
	& {Q''} & {Q' \times_{S' \times_{R}S''} Q''}
	\arrow["g", from=1-1, to=2-1]
	\arrow["f", from=2-1, to=3-1]
	\arrow["{f'}"', from=3-2, to=3-1]
	\arrow["{f''}", from=4-1, to=3-1]
	\arrow[color=red, dotted, from=4-2, to=3-1]
	\arrow[dashed, from=4-2, to=3-2]
	\arrow[dashed, from=4-2, to=4-1]
	\arrow["{g'}", from=4-3, to=4-2]
	\arrow["{g''}"', from=5-2, to=4-2]
	\arrow[color=red, dotted, from=5-3, to=4-2]
	\arrow[dashed, from=5-3, to=4-3]
	\arrow[dashed, from=5-3, to=5-2]
\end{tikzcd}\]

\item Let $\catname{Comp(3-Cospan)}$ denote the subcategory of $(\catname{A}^{\op})^{\times 6}$ the objects of which are pairs of triplets $$\Big(\left(f,f',f''\right),\left(g,g',g''\right)\Big)$$ such that $\left(f,f'\right),\left(f',f''\right),\left(g,g'\right)$ and $\left(g',g''\right)$ are objects of $\catname{Cospan}$ and
$$\text{Cod}(g) = \text{Cod}\left(g'\right) = \text{Cod}\left(g''\right) = \text{Dom}(f) \times_{\text{Cod}(f)}\text{Dom}\left(f'\right) \times_{\text{Cod}(f)} \text{Dom}\left(f''\right).
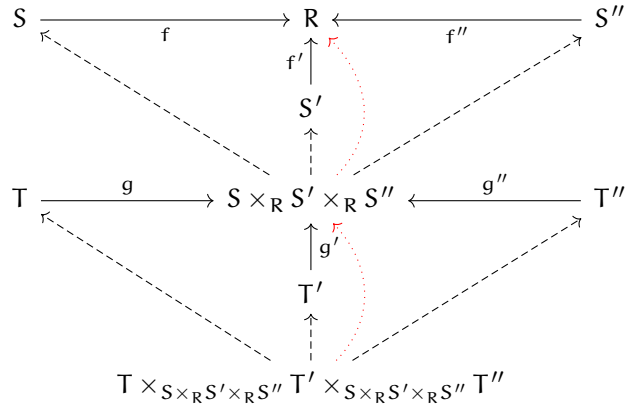
$$

\[\begin{tikzcd}
	S & R & {S''} \\
	& {S'} \\
	T & {S \times_{R} S' \times_{R} S''} & {T''} \\
	& {T'} \\
	& {T \times_{S \times_{R} S' \times_{R} S''} T' \times_{S \times_{R} S' \times_{R} S''} T''}
	\arrow["f"', from=1-1, to=1-2]
	\arrow["{{f''}}", from=1-3, to=1-2]
	\arrow["{f'}", from=2-2, to=1-2]
	\arrow["g", from=3-1, to=3-2]
	\arrow[dashed, from=3-2, to=1-1]
	\arrow[color=red, curve={height=24pt}, dotted, from=3-2, to=1-2]
	\arrow[dashed, from=3-2, to=1-3]
	\arrow[dashed, from=3-2, to=2-2]
	\arrow["{{g''}}"', from=3-3, to=3-2]
	\arrow["{{g'}}"', from=4-2, to=3-2]
	\arrow[dashed, from=5-2, to=3-1]
	\arrow[color=red, curve={height=24pt}, dotted, from=5-2, to=3-2]
	\arrow[dashed, from=5-2, to=3-3]
	\arrow[dashed, from=5-2, to=4-2]
\end{tikzcd}\]

\item Let $\catname{Cospan\Big(Comp(Cospan), Comp\Big)}$ denote the subcategory of $$\catname{Comp(Cospan)} \times \catname{Comp}$$ which is analogous to $$\catname{Cospan\Big(Comp, Comp(Cospan)\Big)},$$ but with the roles of the two pairs being reversed. 
\end{itemize}
\end{notation}

Observe that the categories $\catname{Cospan(Comp)}$ and $\catname{Comp(Cospan)}$ are linked by a functor $$\omega: \catname{Cospan(Comp)} \rightarrow \catname{Comp(Cospan)}$$ defined by $$\omega\left(\left(S \xrightarrow{f}R, T \xrightarrow{g}S\right),\left(S' \xrightarrow{f'}R, T' \xrightarrow{g'}S'\right)\right) = \Big(\left(f,f'\right),\left(g \times_{R}S', S \times_{R}g'\right)\Big).$$ 
The next result appears without proof in \cite{Operads_2013_Hess}. We prove the first part; the second follows by a similar argument.

\begin{corollary} [\cite{Operads_2013_Hess}] \label{nicecorollary}
Let $\Phi, \Phi ', \Psi, \Psi ' \in [\catname{A}^{\op}, \catname{C}]_{\text{mult}}$. Under the equivalence of Lemma \ref{equivalence}, we have that
\begin{enumerate}[1)]
    \item the left Kan extension of $$\catname{Comp(Cospan)} \xrightarrow{\Phi \times \Phi ' \times \Psi \times \Psi '} \catname{C}$$ along the functor $\catname{Comp(Cospan)} \xrightarrow{\phi} \catname{A}^{\op}$ that composes the diagonal maps as shown in Diagram (\ref{diagram_diagonal}) of Notation $\ref{notation}$ corresponds to $$(\mathcal{X}_{\Phi} \square \mathcal{X}_{\Phi'}) \circ (\mathcal{X}_{\Psi} \square \mathcal{X}_{\Psi'}),$$
    \item the left Kan extension of $$\catname{Cospan(Comp)} \xrightarrow{\Phi \times \Psi  \times \Phi ' \times \Psi '} \catname{C}$$ along the functor $\catname{Cospan(Comp)} \xrightarrow{\phi \omega} \catname{A}^{\op}$ corresponds to $$(\mathcal{X}_{\Phi} \circ \mathcal{X}_{\Psi}) \square (\mathcal{X}_{\Phi'} \circ \mathcal{X}_{\Psi'}).$$
\end{enumerate}
\end{corollary}
\begin{proof}
    Let us first understand the functor $\phi: \catname{Comp(Cospan)} \rightarrow \A^\op$ which composes the diagonal maps as shown in Diagram (\ref{diagram_diagonal}) of Notation \ref{notation}. This functor can in fact be seen as the composition of two familiar functors. Given $$\left(\left(S \xrightarrow{f} R, S' \xrightarrow{f'} R\right),\left(T \xrightarrow{g} S \times_R S', T' \xrightarrow{g'} S \times_R S'\right)\right)$$ in $\catname{Comp(Cospan)}$, we apply the fiber product functor $\gamma$ to each pair to get the maps $$S \times_R S' \rightarrow R \quad \text{and} \quad T \times_{S \times_R S'} T' \rightarrow S \times_R S'$$ and then we compose those maps using the composition functor of Lemma \ref{operations_kan}, which we denote by $\epsilon$.

    The following diagram reflects the situation.
\[\begin{tikzcd}
	{\catname{Comp(Cospan)}} && \C \\
	\\
	{\catname{Comp}} & {\A^\op}
	\arrow[""{name=0, anchor=center, inner sep=0}, "{\Phi \times \Phi' \times \Psi \times \Psi'}", from=1-1, to=1-3]
	\arrow["{\gamma \times \gamma}"', from=1-1, to=3-1]
	\arrow["\phi"', from=1-1, to=3-2]
	\arrow["\epsilon"', from=3-1, to=3-2]
	\arrow["{\text{Lan}_\phi \left(\Phi \times \Phi' \times \Psi \times \Psi'\right)}"', dashed, from=3-2, to=1-3]
	\arrow[shorten <=7pt, Rightarrow, from=0, to=3-2]
\end{tikzcd}\]
    By the properties of left Kan extension we have 
    \begin{align*}
        \text{Lan}_\phi \left(\Phi \times \Phi' \times \Psi \times \Psi'\right) &\cong \text{Lan}_\epsilon \left( \text{Lan}_{\gamma \times \gamma} \left(\Phi \times \Phi' \times \Psi \times \Psi' \right)\right) \\ &\cong \text{Lan}_\epsilon \left( \text{Lan}_\gamma \left(\Phi \times \Phi'\right) \times \text{Lan}_\gamma \left(\Psi \times \Psi'\right) \right),
    \end{align*}
    which by Lemma \ref{operations_kan} corresponds to $$(\mathcal{X}_{\Phi} \square \mathcal{X}_{\Phi'}) \circ (\mathcal{X}_{\Psi} \square \mathcal{X}_{\Psi'})$$ via the equivalence of Lemma \ref{equivalence}. This concludes the proof of the first part. The second part of the corollary can be proved in an analogous way.
\end{proof}

By the universal property of the left Kan extension, to construct the natural transformation $\sigma$ in (\ref{natural_sigma}) it suffices to construct a natural transformation 
$$t:(-) \times (-) \times (-) \times (-) \Rightarrow \Big((-) \times (-) \times (-) \times (-)\Big)(23)\omega.$$
The situation is summed up in the following (non-commutative) diagram.
\[\begin{tikzcd}[row sep = large]
	{\catname{Cospan(Comp)}} && {\catname{C}} \\
	\\
	{\catname{Comp(Cospan)}} \\
	\\
	{\catname{A}^{\op}}
	\arrow[""{name=0, anchor=center, inner sep=0}, "{{{\scriptscriptstyle \Phi \times \Psi \times \Phi' \times \Psi'}}}", from=1-1, to=1-3]
	\arrow["{{\scriptscriptstyle \omega}}"', from=1-1, to=3-1]
	\arrow["{{{\scriptscriptstyle \Phi \times \Phi' \times \Psi \times \Psi'}}}"{description, pos=0.4}, from=3-1, to=1-3]
	\arrow["\scriptscriptstyle \varphi"', from=3-1, to=5-1]
	\arrow[""{name=2, anchor=center, inner sep=0}, "{{{\scriptscriptstyle \text{Lan}_{\varphi}\left(\Phi \times \Phi' \times \Psi \times \Psi'\right)}}}"'{pos=0.8}, curve={height=30pt}, from=5-1, to=1-3]
	\arrow["{{{\scriptscriptstyle t_{\Phi, \Psi, \Phi',\Psi'}}}}"{description}, curve={height=12pt}, shorten <=13pt, Rightarrow, from=0, to=3-1]
    \arrow[""{name=1, anchor=center, inner sep=0}, "{{{\scriptscriptstyle \text{Lan}_{\phi \omega} \left(\Phi \times \Psi \times \Phi' \times \Psi'\right)}}}"{description, pos=0.3}, curve={height=-6pt}, from=5-1, to=1-3]
	\arrow["{{\scriptscriptstyle \sigma}}", shorten <=7pt, shorten >=7pt, Rightarrow, from=1, to=2]
\end{tikzcd}\]
To prove that the coherence diagrams of 2-monoidal categories commute, we will also need to understand natural transformations of the following types
\begin{align*}
    \sigma_{A \square C, B \square D, E, F} : \Big((A \square C) \circ (B \square D)\Big) \square (E \circ F) &\Rightarrow \Big((A \square C) \square E\Big) \circ \Big((B \square D) \square F\Big) \\ 
    \sigma_{A, B, C \square E, D \square F} : (A \circ B) \square \Big((C \square E) \circ (D \square F)\Big) &\Rightarrow \Big(A \square (C \square E)\Big) \circ \Big(B \square (D \square F)\Big) \\
    \sigma_{A \circ B, C, D \circ E, F} : \Big((A \circ B) \circ C\Big) \square \Big((D \circ E) \circ F\Big) &\Rightarrow \Big((A \circ B) \square (D \circ E)\Big) \circ (C \square F) \\
    \sigma_{A, B \circ C, D, E \circ F} : \Big(A \circ (B \circ C)\Big) \square \Big(D \circ (E \circ F)\Big) &\Rightarrow (A \square D) \circ \Big((B \circ C) \square (E \circ F)\Big)
\end{align*}
in which we group four of the six symmetric sequences in pairs. For this, we will study the natural transformations $t$ that induce these specific $\sigma$. These correspond to the following
\begin{align*}
    t_{\Phi \square \Phi', \Psi \square \Psi', \Phi'', \Psi''}: (\Phi \square \Phi') \times (\Psi \square \Psi') \times \Phi'' \times \Psi'' &\Rightarrow \Big((\Phi \square \Phi') \times \Phi'' \times (\Psi \square \Psi') \times \Psi''\Big)\omega \\
    t_{\Phi, \Psi,\Phi' \square \Phi'', \Psi' \square \Psi''}: \Phi \times \Psi \times (\Phi' \square \Phi'') \times  (\Psi' \square \Psi'') &\Rightarrow \Big(\Phi \times (\Phi' \square \Phi'') \times \Psi \times  (\Psi' \square \Psi'')\Big)\omega \\
    t_{\Phi \circ \Psi, \Phi', \Psi' \circ \Phi'', \Psi''}: (\Phi \circ \Psi) \times \Phi' \times  (\Psi' \circ \Phi'') \times \Psi'' &\Rightarrow \Big((\Phi \circ \Psi) \times  (\Psi' \circ \Phi'') \times \Phi' \times \Psi''\Big)\omega \\
    t_{\Phi,  \Psi \circ \Phi', \Psi', \Phi'' \circ \Psi''}: \Phi \times (\Psi \circ \Phi') \times \Psi' \times (\Phi'' \circ \Psi'') &\Rightarrow \Big(\Phi \times \Psi' \times (\Psi \circ \Phi') \times (\Phi'' \circ \Psi'')\Big)\omega 
\end{align*}
where, by abuse of notation, $\Phi \circ \Psi$ denotes the Kan extension from part 2) of Lemma \ref{operations_kan} and $\Phi \square \Psi$ denotes the one from part 1).

We want to prove the commutativity of the diagrams in Definition \ref{definition_2-monoidal} by passing to the category of multiplicative functors and establishing it pointwise for any choice of objects in $\A^{\op}$. However, the four transformations above cannot be directly inserted into these diagrams, since their domain is a product of four multiplicative functors, whereas the other parts of the diagrams involve six. We therefore need to freely interchange between the four-functor and six-functor cases. This is achieved via natural transformations corresponding to the four types of $t$ above, as explained in their construction. We denote these by $t^{(1)}, t^{(2)}, t^{(3)}$, and $t^{(4)}$ in the same order and their corresponding induced $\sigma$ by $\sigma^{(1)}, \sigma^{(2)}, \sigma^{(3)}$, and $\sigma^{(4)}$, respectively. The discussion involves the categories discussed in the second part of Notation \ref{notation}.

\begin{notation} Before doing so, we need additional combinatorial machinery in the form of significant functors between the involved categories.
\begin{itemize}
\item The categories $\catname{Cospan\Big(Comp(Cospan), Comp\Big)}$ and $\catname{Comp(3-Cospan)}$ are linked by a functor 
$$\omega^{(1)}: \catname{Cospan\Big(Comp(Cospan), Comp\Big)} \rightarrow \catname{Comp(3-Cospan)}$$ that takes $$\Bigg(\left(\left(S \xrightarrow{f}R, S' \xrightarrow{f'} R\right), \left(T \xrightarrow{g} S \times_{R}S', T' \xrightarrow{g'} S \times_{R}S'\right)\right), \left(S'' \xrightarrow{f''} R, T'' \xrightarrow{g''} S''\right)\Bigg)$$ to $$\Big(\left(f,f',f''\right),\left(g \times_{R}S'', g' \times_{R}S'',S \times_{R} S' \times_{R} g''\right)\Big).$$

\item The categories $\catname{Cospan\Big(Comp,Comp(Cospan)\Big)}$ and $\catname{Comp(3-Cospan)}$ are linked by a functor 
$$\omega^{(2)}: \catname{Cospan\Big(Comp,Comp(Cospan)\Big)} \rightarrow \catname{Comp(3-Cospan)}$$ that takes $$\Bigg(\left(S \xrightarrow{f}R, T \xrightarrow{g} S\right), \left(\left(S' \xrightarrow{f'}R, S'' \xrightarrow{f''}R\right),\left(T' \xrightarrow{g'} S' \times_{R}S'', T'' \xrightarrow{g''} S' \times_{R}S''\right)\right)\Bigg)$$ to $$\Big(\left(f,f',f''\right),\left(g \times_{R}S' \times_{R}S'',S \times_{R} g',S \times_{R} g''\right)\Big).$$

\item The categories $\catname{Cospan(3-Comp)}$ and $\catname{Comp\Big(Cospan(Comp), Cospan\Big)}$ are linked by a functor 
$$\omega^{(3)}: \catname{Cospan(3-Comp)} \rightarrow \catname{Comp\Big(Cospan(Comp), Cospan\Big)}$$ that takes $\Bigg(\left(S \xrightarrow{f}R, T \xrightarrow{g} S, Q \xrightarrow{h} T\right),\left(S' \xrightarrow{f'}R, T' \xrightarrow{g'} S', Q' \xrightarrow{h} T'\right)\Bigg)$ to $$\Bigg(\Big(\left(f,g\right),\left(f',g'\right)\Big),\left(h \times_{R} T', T \times_{R} h'\right)\Bigg).$$

\item The categories $\catname{Cospan(3-Comp)}$ and $\catname{Comp\Big(Cospan, Cospan(Comp)\Big)}$ are linked by a functor 
$$\omega^{(4)}: \catname{Cospan(3-Comp)} \rightarrow \catname{Comp\Big(Cospan, Cospan(Comp)\Big)}$$ that takes $\Bigg(\left(S \xrightarrow{f}R, T \xrightarrow{g} S, Q \xrightarrow{h} T\right),\left(S' \xrightarrow{f'}R, T' \xrightarrow{g'} S', Q' \xrightarrow{h} T'\right)\Bigg)$ to $$\Bigg(\left(f,f'\right),\Big(\left(g \times_{R} S', h \times_{R} S'\right),\left(S \times_{R} g', S \times_{R} h'\right)\Big)\Bigg).$$

\item The categories $\catname{Cospan(3-Comp)}$ and $\catname{Cospan(Comp)}$ are linked by two functors
$$\partial,\partial':\catname{Cospan(3-Comp)} \rightarrow \catname{Cospan(Comp)}$$ that take $\Bigg(\left(S \xrightarrow{f}R, T \xrightarrow{g} S, Q \xrightarrow{h} T\right),\left(S' \xrightarrow{f'}R, T' \xrightarrow{g'} S', Q' \xrightarrow{h} T'\right)\Bigg)$ to $$\Big(\left(fg,h\right),\left(f'g',h'\right)\Big)$$ and $\Bigg(\left(S \xrightarrow{f}R, T \xrightarrow{g} S, Q \xrightarrow{h} T\right),\left(S' \xrightarrow{f'}R, T' \xrightarrow{g'} S', Q' \xrightarrow{h} T'\right)\Bigg)$ to $$\Big(\left(f,gh\right),\left(f',g'h'\right)\Big).$$

\item The categories $\catname{Cospan\Big(Comp(Cospan),Comp\Big)}$ and $\catname{Cospan(Comp)}$ are linked by a functor 
$$\Delta:\catname{Cospan\Big(Comp(Cospan),Comp\Big)} \rightarrow \catname{Cospan(Comp)}$$ that takes $$\Bigg(\left(\left(S \xrightarrow{f}R, S' \xrightarrow{f'} R\right), \left(T \xrightarrow{g} S \times_{R}S', T' \xrightarrow{g'} S \times_{R}S'\right)\right), \left(S'' \xrightarrow{f''} R, T'' \xrightarrow{g''} S''\right)\Bigg)$$ to $$\Big(\left(f \times_{R} f', g \times_{S \times_{R}S'}g'\right),\left(f'',g''\right)\Big).$$

\item The categories $\catname{Cospan\Big(Comp, Comp(Cospan)\Big)}$ and $\catname{Cospan(Comp)}$ are linked by a functor 
$$\Delta':\catname{Cospan\Big(Comp,Comp(Cospan)\Big)} \rightarrow \catname{Cospan(Comp)}$$ that takes 
$$\Bigg(\left(S \xrightarrow{f}R, T \xrightarrow{g} S\right), \left(\left(S' \xrightarrow{f'}R, S'' \xrightarrow{f''}R\right),\left(T' \xrightarrow{g'} S' \times_{R}S'', T'' \xrightarrow{g''} S' \times_{R}S''\right)\right)\Bigg)$$ to $$\Big(\left(f,g\right),\left(f' \times_{R} f'', g' \times_{ S' \times_{R}S''}g''\right)\Big).$$

\item The categories $\catname{Comp(3-Cospan)}$ and $\catname{A}^{\op}$ are linked by a functor
$$\phi^{(1/2)}:\catname{Comp(3-Cospan)} \rightarrow \catname{A}^{\op}$$ that takes 
$$\Bigg(\left(S \xrightarrow{f}R, S' \xrightarrow{f'} R\right),\left(T \xrightarrow{g}S\times_{R}S', T' \xrightarrow{g'} S\times_{R}S'\right),\left(Q \xrightarrow{h}T\times_{R}T', Q' \xrightarrow{h'} T\times_{R}T'\right)\Bigg)$$ to $$\left(f \times_{R} f'\right) \circ \left(g \times_{S\times_{R}S'}g'\right) \circ \left(h \times_{T\times_{R}T'}h'\right).$$ 

\item The categories $\catname{Comp\Big(Cospan(Comp),Cospan\Big)}$ and $\catname{A}^{\op}$ are linked by a functor 
$$\phi^{(3)}:\catname{Comp\Big(Cospan(Comp),Cospan\Big)} \rightarrow \catname{A}^{\op}$$ that takes 
$$\Bigg(\left(\left(S \xrightarrow{f}R, T \xrightarrow{g}S\right),\left(S' \xrightarrow{f'}R, T' \xrightarrow{g'}S'\right)\right),\left(Q \xrightarrow{h}T \times_{R}T', Q' \xrightarrow{h'}T \times_{R}T'\right)\Bigg)$$ to $$\left(fg \times_{R} f'g'\right) \circ \left(h \times_{T \times_{R}T'} h'\right).$$

\item The categories $\catname{Comp\Big(Cospan, Cospan(Comp)\Big)}$ and $\catname{A}^{\op}$ are linked by a functor 
$$\phi^{(4)}:\catname{Comp\Big(Cospan, Cospan(Comp)\Big)} \rightarrow \catname{A}^{\op}$$ that takes 
$$\Bigg(\left(S \xrightarrow{f}R, S' \xrightarrow{f'}R\right), \left(\left(T \xrightarrow{g}S \times_{R}S', Q \xrightarrow{h} T\right),\left(T' \xrightarrow{g'}S \times_{R}S', Q' \xrightarrow{h'} T'\right)\right)\Bigg)$$ to $$\left(f \times_{R} f'\right) \circ \left(gh \times_{S \times_{R}S'}g'h'\right).$$
\end{itemize}
\end{notation}

Let us explain the construction and motivation of $t^{(4)}$, which must intuitively coincide with $t_{\Phi, \Psi \circ \Phi', \Phi'' \circ \Psi''}$. The other natural transformations mentioned above can be constructed in the same way by using the different categories and functors introduced (the ones with corresponding indices), so we will simply give their expressions at the end.

Consider the following diagram.
\[\begin{tikzcd}[column sep=tiny]
	{\catname{Cospan(3-Comp)}} && {\catname{Comp(Cospan, Cospan(Comp))}} \\
	{\catname{Cospan(Comp)}} & {\catname{Comp(Cospan)}} & {\catname{A}^\op}
	\arrow["{{\omega^{(4)}}}", from=1-1, to=1-3]
	\arrow["{{\partial'}}"', from=1-1, to=2-1]
	\arrow["{{\varphi^{(4)}}}", from=1-3, to=2-3]
	\arrow["\omega"', from=2-1, to=2-2]
	\arrow["\varphi"', from=2-2, to=2-3]
\end{tikzcd}\]
It is straightforward to check that it is commutative by construction. Hence, for any multiplicative functors $\Phi, \Psi, \Phi', \Psi', \Phi'', \Psi''$ there is a unique natural isomorphism $$\text{Lan}_{\varphi^{(4)}\omega^{(4)}}(\Phi \times \Psi \times \Phi' \times \Psi' \times \Phi'' \times \Psi'') \cong \text{Lan}_{\varphi \omega \partial'}(\Phi \times \Psi \times \Phi' \times \Psi' \times \Phi'' \times \Psi'').$$

\begin{lemma}\label{lemma_left_kan_1}
    There is a unique natural isomorphism $$\text{Lan}_{\varphi^{(4)}\omega^{(4)}}(\Phi \times \Psi \times \Phi' \times \Psi' \times \Phi'' \times \Psi'') \cong \text{Lan}_{\phi\omega} \Big(\Phi \times (\Psi \circ \Phi') \times \Psi' \times (\Phi'' \circ \Psi'')\Big).$$
\end{lemma}
\begin{proof}
    By elementary properties of left Kan extensions we get $$\text{Lan}_{\varphi \omega \partial'}(\Phi \times \Psi \times \Phi' \times \Psi' \times \Phi'' \times \Psi'') \cong \text{Lan}_{\varphi \omega}\Big(\text{Lan}_{\partial'}(\Phi \times \Psi \times \Phi' \times \Psi' \times \Phi'' \times \Psi'')\Big).$$ Now, let us analyze what the functor $\partial'$ does on each of the two components of $\catname{3-Comp}$ in $\catname{Cospan(3-Comp)} \subseteq \catname{(3-Comp)}^{\times 2}$. Clearly, $\partial'$ acts separately on each of the copies of $\catname{3-Comp}$, so it is of the form $\partial' = \partial_0' \times \partial_0'$, where $\partial_0':\catname{3-Comp} \rightarrow \catname{Comp}$ composes the last two maps. Hence, $$\text{Lan}_{\partial'}(\Phi \times \Psi \times \Phi' \times \Psi' \times \Phi'' \times \Psi'') \cong \text{Lan}_{\partial_0'}(\Phi \times \Psi \times \Phi') \times  \text{Lan}_{\partial_0'}(\Psi' \times \Phi'' \times \Psi'')$$ because left Kan extensions preserves products. Now, $\partial_0'$ acts identically on the first copy of $\catname{A}^\op$ in $\catname{3-Comp}$ and acts by composition on the last two copies, so $$\text{Lan}_{\partial_0'}(\Phi \times \Psi \times \Phi') \cong \Phi \times (\Psi \circ \Phi')$$ and $$\text{Lan}_{\partial_0'}(\Psi' \times \Phi'' \times \Psi'') \cong \Psi' \times (\Phi'' \circ \Psi'').$$
    All mentioned natural isomorphisms are unique. By combining those isomorphisms with the observation above the lemma, we obtain the desired conclusion.
\end{proof}

\begin{lemma}\label{lemma_left_kan_2}
    There is a unique natural isomorphism 
    $$\text{Lan}_{\varphi^{(4)}}(\Phi \times \Psi \times \Phi' \times \Psi' \times \Phi'' \times \Psi'') \cong \text{Lan}_{\varphi}\Big(\Phi \times \Psi' \times (\Psi \circ \Phi') \times (\Phi'' \circ \Psi'')\Big).$$
\end{lemma}
\begin{proof}
To this end, let us analyze what the functor $\varphi^{(4)}$ does. We look at $$\catname{Comp\Big(Cospan, Cospan(Comp)\Big)}$$ as a subcategory of $$\catname{Cospan} \times \catname{Cospan(Comp)}.$$ Clearly, if we look at the functor $$F: \catname{Cospan(Comp)} \rightarrow \catname{Cospan}$$ that composes the two pairs, restriction and corestriction yield a commutative diagram.
\[\begin{tikzcd}
	{\catname{Comp\Big(Cospan, Cospan(Comp)\Big)}} && {\catname{Comp(Cospan)}} \\
	\\
	{\catname{A}^\op}
	\arrow["{\id_{\catname{Cospan}} \times F}", from=1-1, to=1-3]
	\arrow["{\varphi^{(4)}}"', from=1-1, to=3-1]
	\arrow["\varphi", from=1-3, to=3-1]
\end{tikzcd}\]
Again by properties of left Kan extensions we infer 
\begin{align*}
    &\text{Lan}_{\varphi^{(4)}}(\Phi \times \Psi \times \Phi' \times \Psi' \times \Phi'' \times \Psi'') \\ &\cong \text{Lan}_{\varphi}\Big(\text{Lan}_{\id_{\catname{Cospan}} \times F}(\Phi \times \Psi \times \Phi' \times \Psi' \times \Phi'' \times \Psi'')\Big) \\
    &\cong \text{Lan}_{\varphi}\Big(\text{Lan}_{\id}(\Phi \times \Psi') \times \text{Lan}_{F}(\Psi \times \Phi' \times \Phi'' \times \Psi'')\Big) \\
    &\cong \text{Lan}_{\varphi}\Big(\Phi \times \Psi' \times \text{Lan}_{F}(\Psi \times \Phi' \times \Phi'' \times \Psi'')\Big).
\end{align*}
On each of the two copies of $\catname{Comp}$ in $\catname{Cospan(Comp)}$, $F$ acts by composition and therefore $$\text{Lan}_{F}(\Psi \times \Phi' \times \Phi'' \times \Psi'') \cong (\Psi \circ \Phi') \times (\Phi'' \circ \Psi'').$$ 
We are done because all of the above are unique natural isomorphisms.
\end{proof}

Consider the following diagram.
\[\begin{tikzcd}[column sep = huge, row sep = large]
	{\catname{A}^\op} \\
	\\
	{\catname{Comp(Cospan)}} \\
	\\
	{\catname{Cospan(Comp)}} \\
	\\
	{\catname{Cospan(3-Comp)}} && {\catname{C}} \\
	\\
	{\catname{Comp\Big(Cospan,Cospan(Comp)\Big)}} \\
	\\
	{\catname{A}^{\op}}
	\arrow[color=blue,""{name=0, anchor=center, inner sep=0}, "{\scalebox{0.6}{$\text{Lan}_\phi \big(\Phi \times \Psi' \times (\Psi \circ \Phi') \times (\Phi'' \circ \Psi'')\big)$}}"{pos=0.8}, curve={height=-30pt}, from=1-1, to=7-3]
	\arrow[color=red,""{name=1, anchor=center, inner sep=0}, "{\scalebox{0.6}{$\text{Lan}_{\phi\omega} \big(\Phi \times (\Psi \circ \Phi') \times \Psi' \times (\Phi'' \circ \Psi'')\big)$}}"{description, pos=0.4}, curve={height=-6pt}, from=1-1, to=7-3]
	\arrow["{{{\scriptscriptstyle \phi}}}", from=3-1, to=1-1]
	\arrow["{{{\scriptscriptstyle \omega}}}", from=5-1, to=3-1]
	\arrow[""{name=2, anchor=center, inner sep=0}, "{{{{{\scalebox{0.6}{$\Phi \times (\Psi \circ \Phi') \times \Psi' \times (\Phi'' \circ \Psi'')$}}}}}}"', from=5-1, to=7-3]
	\arrow["{{{\scriptscriptstyle \partial'}}}", from=7-1, to=5-1]
	\arrow[""{name=3, anchor=center, inner sep=0}, "{{{{{\scalebox{0.6}{$\Phi \times \Psi \times \Phi' \times \Psi' \times \Phi'' \times \Psi''$}}}}}}"{description}, from=7-1, to=7-3]
	\arrow["{{{\scriptscriptstyle \omega^{(4)}}}}"', from=7-1, to=9-1]
	\arrow["{{{{\scalebox{0.6}{$\Phi \times \Psi' \times \Psi \times \Phi' \times \Phi'' \times \Psi''$}}}}}"{description, pos=0.3}, from=9-1, to=7-3]
	\arrow["{{{\scriptscriptstyle \phi^{(4)}}}}"', from=9-1, to=11-1]
	\arrow[color=blue,""{name=4, anchor=center, inner sep=0}, "{{{{{\scalebox{0.6}{$\text{Lan}_{\varphi^{(4)}} \big(\Phi \times \Psi' \times \Psi \times \Phi' \times \Phi'' \times \Psi''\big)$}}}}}}"'{pos=0.8}, curve={height=30pt}, from=11-1, to=7-3]
	\arrow[color=red,""{name=5, anchor=center, inner sep=0}, "{{{{\scalebox{0.6}{$\text{Lan}_{\phi^{(4)} \omega^{(4)}} \big(\Phi \times \Psi \times \Phi' \times \Psi' \times \Phi'' \times \Psi''\big)$}}}}}"{description, pos=0.3}, curve={height=6pt}, from=11-1, to=7-3]
	\arrow[color=green,"{{{\scriptscriptstyle \sigma}}}"', shorten <=4pt, shorten >=4pt, Rightarrow, from=1, to=0]
	\arrow[color=orange,"{\scriptscriptstyle t}"{description}, curve={height=-6pt}, shorten <=17pt, Rightarrow, from=2, to=3-1]
 \arrow["{{{{{\scalebox{0.6}{$\Phi \times \Psi' \times (\Psi \circ \Phi') \times (\Phi'' \circ \Psi'')$}}}}}}"{description}, from=3-1, to=7-3]
	\arrow[color=orange,"{{{{{{\scriptscriptstyle t^{(4)}}}}}}}"{description}, curve={height=12pt}, shorten <=17pt, Rightarrow, from=3, to=9-1]
	\arrow[color=green,"{{{\scriptscriptstyle \sigma^{(4)}}}}", shorten <=4pt, shorten >=4pt, Rightarrow, from=5, to=4]
\end{tikzcd}\]

By Lemmas \ref{lemma_left_kan_1} and \ref{lemma_left_kan_2} (the relevant left Kan extensions are color-coded in the diagram), to obtain $\sigma^{(4)}$ coinciding with $\sigma$ up to the two unique natural isomorphisms, it suffices to define $t^{(4)}$ so that it agrees with $t$ up to the left Kan extensions yielding the composition on the two indicated pairs of functors. Such a $t^{(4)}$ induces $\sigma^{(4)}$ by the uniqueness of left Kan extensions up to natural transformation. The fact that $t^{(4)}$ and $t$ agree follows directly from their construction, which we provide below, along with the expressions for $t^{(1)}, t^{(2)}$, and $t^{(3)}$.

\begin{notation}
For $\Phi \in [\catname{A}^{\op}, \catname{C}]_{\text{mult}}$ and $T \xrightarrow{g} S, S \rightarrow R$, and $S' \rightarrow R$ in $\catname{A}^\op$, we construct a natural morphism $$\delta_{\Phi} : \Phi(g)  \rightarrow \Phi\left(g \times_R S'\right) $$ in the following way. We have the isomorphisms
$$\Phi(g) \cong \prod_{s \in S} \Phi\left(g^{-1}(\{s\}) \rightarrow \{s\}\right)$$ and
$$\Phi\left(g \times_R S'\right) \cong \prod_{(s,s') \in S \times_{R}S'} \Phi\left(g^{-1}(\{s\}) \times \{s'\} \rightarrow \{(s,s')\}\right),$$
and morphisms in $\catname{A}^\op$ for every pair $(s,s')$ given by 
$$\left(\id \times \{s'\}, \iota \right): \left(g^{-1}(\{s\}) \rightarrow \{s\}\right) \rightarrow \left(g^{-1}(\{s\}) \times \{s'\} \rightarrow \{(s,s')\}\right).$$
Let then $$\delta_{\Phi} = \prod_{(s,s') \in S \times_{R}S'} \Phi\left((\id \times \{s'\}, \iota)^{-1}\right).$$
Similarly, given additionally $S'' \rightarrow R$ in $\catname{A}^{\op}$, we construct a natural morphism $${\widetilde \delta}_{\Phi}: \Phi(g) \rightarrow \Phi\left(g \times_{R}S' \times_{R}S''\right),$$ which is simply an iterated version of $\delta_{\Phi}$.
\end{notation}

Now, given $\left(S \xrightarrow{f}R, T \xrightarrow{g}S\right),\left(S' \xrightarrow{f'}R, T' \xrightarrow{g'}S'\right)$, we define $$t: \Phi(f) \times \Psi(g) \times \Phi'\left(f'\right) \times \Psi'\left(g'\right) \rightarrow \Phi(f) \times \Phi'\left(f'\right) \times \Psi\left(g \times_{R}S'\right) \times \Psi'\left(S \times_{R}g'\right)$$ by $$t\left(x,y,x',y'\right) = \Big(x,x',\delta_{\Psi}(y), \delta_{\Psi'}\left(y'\right)\Big).$$

Omitting domains and codomains (which correspond to the respective notations for the $\omega^{(-)}$ functors), we also define
\begin{align*}
    t^{(1)}\left(x,x',y,y',x'',y''\right) &= \Big(x,x',x'',\delta_{\Psi}(y), \delta_{\Psi'}\left(y'\right), {\widetilde \delta}_{\Psi''}\left(y''\right)\Big) \\
    t^{(2)}\left(x,y,x',x'',y',y''\right) &=\Big(x,x',x'',{\widetilde \delta}_{\Psi}(y), \delta_{\Psi'}\left(y'\right), \delta_{\Psi''}\left(y''\right)\Big) \\
    t^{(3)}\left(x,y,x',y',x'',y''\right) &= \Big(x,y,y',x'',{\widetilde \delta}_{\Phi'}\left(x'\right),{\widetilde \delta}_{\Psi''}\left(y''\right)\Big) \\
    t^{(4)}\left(x,y,x',y',x'',y''\right) &= \Big(x,y',\delta_{\Psi}(y), \delta_{\Phi'}\left(x'\right), \delta_{\Phi''}\left(x''\right), \delta_{\Psi''}\left(y''\right)\Big).
\end{align*}

Let us analyze why $t$ and $t^{(4)}$ correspond to one another under the identifications mentioned above. The domain of $t$ is
$$\Phi(f) \times \left(\Psi \circ \Phi'\right)\left(g \circ h\right) \times \Psi'\left(f'\right) \times \left(\Phi'' \circ \Psi''\right)\left(g' \circ h'\right)$$
and its codomain is
$$\Phi(f) \times \Psi'\left(f'\right) \times \left(\Psi \circ \Phi'\right)\Big((g \circ h) \times_{R}S'\Big) \times \left(\Phi'' \circ \Psi''\right)\Big(S \times_{R} \left(g' \circ h'\right)\Big) $$
whereas the domain of $t^{(4)}$ is
$$\Phi(f) \times \Psi(g) \times \Phi'(h) \times \Psi'\left(f'\right) \times \Phi''\left(g'\right) \times \Psi''\left(h'\right)$$
and its codomain is
$$\Phi(f) \times \Psi'\left(f'\right) \times \Psi\left(g \times_{R}S'\right) \times \Phi'\left(h \times_{R}S'\right) \times \Phi''\left(S \times_{R}S'\right) \times\Psi''\left(S \times_{R}h'\right).$$
The components corresponding to $\Phi$ and $\Psi'$ are mapped identically to the first two components of the codomain in both transformations. Up to the left Kan extensions yielding the composition of symmetric sequences on the last two pairs of functors, $t$ and $t^{(4)}$ coincide, since the composition of fiber products is the fiber product of the composition. This gives the desired identification.

\subsection{Proof of the main theorem}
We are now ready to prove Theorem \ref{2-monoidal}.
\begin{proof}[Proof of Theorem \ref{2-monoidal}]
The arguments presented hold in the concrete setting but carry over completely to the non-concrete case.
Observe that in the passage from compositions (whether it is the composition or the matrix monoidal structure) of symmetric sequences to external products of functors, the associator becomes a change of parenthesization. 

By the equivalence proved in Lemma \ref{equivalence} and Corollary \ref{nicecorollary}, as well as the constructions of the designated natural transformations, the coherence diagrams (\ref{diagram_2-monoidal_associativity_1}) and (\ref{diagram_2-monoidal_associativity_2}) of 2-monoidal categories in Definition \ref{definition_2-monoidal} commute if and only if the diagrams on the following page commute. The first one for all $$((f,g),(f',g'),(f'',g'')) \in \catname{3-Cospan(Comp)}$$ and the second one for all $$ ((f,g,h),(f',g',h')) \in \catname{Cospan(3-Comp)}.$$

\newpage
\begin{landscape} 
\vspace*{\fill}
\adjustbox{width = {20cm}}{  
\centering
$\begin{tikzcd}
	{\Phi(f) \times \Psi(g) \times \Phi'(f') \times \Psi'(g') \times \Phi''(f'') \times \Psi''(g'')} & {\Phi(f) \times \Psi(g) \times \Phi'(f') \times \Phi''(f'') \times \Psi'(g' \times_R S'') \times \Psi''(S' \times_R g'')} \\
	\\
	{\Phi(f) \times \Phi'(f') \times \Psi(g \times_R S')  \times \Psi'(S \times_R g') \times \Phi''(f'') \times \Psi''(g'')} & {\Phi(f) \times \Phi'(f') \times \Phi''(f'') \times \Psi(g \times_R S' \times_R S'') \times \Psi'(S \times_R g' \times_R S'') \times \Psi''(S \times_R S' \times_R g'')} \\
	{(x,y,x',y',x'',y'')} & {(x,y,x',x'', \delta_{\Psi'}(y'), \delta_{\Psi''}(y''))} \\
	& {(x,x',x'',{\widetilde \delta}_\Psi(y),\delta_{\Psi'}(\delta_{\Psi'}(y')),\delta_{\Psi''}(\delta_{\Psi''}(y'')))} \\
	{(x,x',\delta_\Psi(y),\delta_{\Psi'}(y'),x'',y'')} & {(x,x',x'',\delta_\Psi(\delta_\Psi(y)),\delta_{\Psi'}(\delta_{\Psi'}(y')),{\widetilde \delta}_{\Psi''}(y''))}
	\arrow["{{{\id \times \id \times t}}}", from=1-1, to=1-2]
	\arrow["{{{t \times \id \times \id}}}"', from=1-1, to=3-1]
	\arrow["{{{t^{(2)}}}}", from=1-2, to=3-2]
	\arrow["{{{t^{(1)}}}}"', from=3-1, to=3-2]
	\arrow[maps to, from=4-1, to=4-2]
	\arrow[maps to, from=4-1, to=6-1]
	\arrow[maps to, from=4-2, to=5-2]
	\arrow[maps to, from=6-1, to=6-2]
	\arrow[equals, from=6-2, to=5-2]
\end{tikzcd}$
} \vspace{1.85cm} \\
\adjustbox{width = {20cm}}{  
\centering
$\begin{tikzcd}
	{\Phi(f) \times \Psi(g) \times \Phi'(h) \times \Psi'(f') \times \Phi''(g') \times \Psi''(h')} & {\Phi(f) \times \Psi'(f') \times \Psi(g \times_R S') \times \Phi'(h \times_R S') \times \Phi''(S \times_R g') \times \Psi''(S \times_R h')} \\
	\\
	{\Phi(f) \times \Psi(g) \times \Psi'(f') \times \Phi''(g') \times \Phi'(h \times_R S' \times_R S'') \times \Psi''(S \times_R h' \times_R S'')} & {\Phi(f) \times \Psi'(f') \times \Psi(g \times_R S') \times \Phi''(S \times_R g') \times \Phi'(h \times_R S' \times_R S'') \times \Psi''(S \times_R h' \times_R S'')} \\
	{(x,y,x',y',x'',y'')} & {(x,y',\delta_{\Psi}(y),\delta_{\Phi'}(x'),\delta_{\Phi''}(x''),\delta_{\Psi''}(y''))} \\
	& {(x,y',\delta_{\Psi}(y),\delta_{\Phi''}(x''),\delta_{\Phi'}(\delta_{\Phi'}(x')),\delta_{\Psi''}(\delta_{\Psi''}(y'')))} \\
	{(x,y,y',x'',{\widetilde \delta}_{\Phi'}(x'), {\widetilde \delta}_{\Psi''}(y''))} & {(x,y',\delta_{\Psi}(y),\delta_{\Phi''}(x''),{\widetilde \delta}_{\Phi'}(x'), {\widetilde \delta}_{\Psi''}(y''))}
	\arrow["{{t^{(4)}}}", from=1-1, to=1-2]
	\arrow["{{t^{(3)}}}"', from=1-1, to=3-1]
	\arrow["{{\id \times \id \times t}}", from=1-2, to=3-2]
	\arrow["{{t \times \id \times \id}}"', from=3-1, to=3-2]
	\arrow[maps to, from=4-1, to=4-2]
	\arrow[maps to, from=4-1, to=6-1]
	\arrow[maps to, from=4-2, to=5-2]
	\arrow[maps to, from=6-1, to=6-2]
	\arrow[equals, from=6-2, to=5-2]
\end{tikzcd}$
} \\
\vspace*{\fill}
\end{landscape}
\newpage

As shown in the mapping diagrams, these do indeed commute. Moreover, since the two monoidal laws share the same unit, $\Delta_{\mathcal{J}}$ and $\mu_{\mathcal{J}}$ are both the coherence isomorphisms, and $$\iota_{\mathcal{J}} = \epsilon_{\mathcal{J}} = \id_{\mathcal{J}}.$$ The four unitality diagrams (\ref{diagram_2-monoidal_unitality_1}) and (\ref{diagram_2-monoidal_unitality_2}) can be verified to commute as well. Finally, as the monoidal unit is a bimonoid, the compatibility of units holds. This concludes the proof.
\end{proof}

\section{Consequences of the 2-monoidal structure on symmetric sequences}
\label{section_implications}
We now study some consequences of the 2-monoidal structure of the category of symmetric sequences. 

\subsection{Monoids, comonoids, bimonoids, and double (co)monoids}
The next corollary is a consequence of Proposition 6.35 of \cite{Aguiar2010}, which provides an extensive source of Day-operads.

\begin{corollary}[\cite{Aguiar2010}] \label{corollary_source_Day_operads}
    If $\mathcal{P}$ and $\mathcal{Q}$ are Day-operads, i.e., monoids in the monoidal category of symmetric sequences equipped with the matrix monoidal structure $(\catname{sSeq}, \square, \mathcal{J})$, then $\mathcal{P} \circ \mathcal{Q}$ is a Day-operad as well.
\end{corollary}

Using the theory developed in \cite{Aguiar2010}, we can say more about the relations between the monoids and comonoids of the two monoidal structures. We briefly define the notions of bimonoids and double (co)monoids.

\begin{definition}
    Let $(\catname{C},\square, I, \circ, J)$ be a 2-monoidal category. A \textit{bimonoid} in $\catname{C}$ is a quintuple $(H,\mu,\iota,\Delta,\epsilon)$ where $(H,\mu,\iota)$ is a monoid in $(\C,\square,I)$, $(H,\Delta,\epsilon)$ is a comonoid in $(\C,\circ,J)$, and the two structures are compatible in the sense that the following four diagrams commute.
\[\begin{tikzcd}
	{(H \circ H)\square (H \circ H)} && {(H\square H)\circ (H \square H)} \\
	{H \square H} & H & {H \circ H}
	\arrow["{\zeta_{H,H,H,H}}", from=1-1, to=1-3]
	\arrow["{\mu \circ \mu}", from=1-3, to=2-3]
	\arrow["{\Delta \square \Delta}", from=2-1, to=1-1]
	\arrow["\mu"', from=2-1, to=2-2]
	\arrow["\Delta"', from=2-2, to=2-3]
\end{tikzcd}\]
\[\begin{tikzcd}
	{H \square H} & {J \square J} && I & H \\
	H & J && {I \circ I} & {H \circ H}
	\arrow["{\epsilon \square \epsilon}", from=1-1, to=1-2]
	\arrow["\mu"', from=1-1, to=2-1]
	\arrow["{\mu_J}", from=1-2, to=2-2]
	\arrow["\iota", from=1-4, to=1-5]
	\arrow["{\Delta_I}"', from=1-4, to=2-4]
	\arrow["\Delta", from=1-5, to=2-5]
	\arrow["\epsilon"', from=2-1, to=2-2]
	\arrow["{\iota \circ \iota}"', from=2-4, to=2-5]
\end{tikzcd}\]
\[\begin{tikzcd}
	& H \\
	I && J
	\arrow["\epsilon", from=1-2, to=2-3]
	\arrow["\iota", from=2-1, to=1-2]
	\arrow["{\iota_J = \epsilon_I}"', from=2-1, to=2-3]
\end{tikzcd}\]
    A morphism of bimonoids is a morphism of the underlying monoids and comonoids. 
    
    The category of bimonoids in $(\catname{C},\square, I, \circ, J)$ is denoted by $\text{Bimon}(\catname{C},\square, I, \circ, J)$.
\end{definition}

\begin{definition}
    Let $(\catname{C},\square, I, \circ, J)$ be a 2-monoidal category. A \textit{double monoid} in $\catname{C}$ is an object $A$ equipped with morphisms
\[\begin{tikzcd}[column sep=scriptsize]
	{A \square A} & A && I & A && {A \circ A} & A && {J } & A,
	\arrow[from=1-1, to=1-2]
	\arrow[from=1-4, to=1-5]
	\arrow[from=1-7, to=1-8]
	\arrow[from=1-10, to=1-11]
\end{tikzcd}\]
    which turn $A$ into a monoid in both $(\catname{C},\square,I)$ and $(\catname{C},\circ,J)$, and make the following diagrams commute.
\[\begin{tikzcd}
	{(A \circ A) \square (A \circ A)} && {(A \square A) \circ (A \square A)} \\
	{A \square A} & A & {A \circ A}
	\arrow["{\zeta_{A,A,A,A}}", from=1-1, to=1-3]
	\arrow[from=1-1, to=2-1]
	\arrow[from=1-3, to=2-3]
	\arrow[from=2-1, to=2-2]
	\arrow[from=2-3, to=2-2]
\end{tikzcd}\]
\[\begin{tikzcd}
	{J \square J} & {A \square A} && {I \circ I} & {A \circ A} \\
	J & A && I & A
	\arrow[from=1-1, to=1-2]
	\arrow["{\mu_J}"', from=1-1, to=2-1]
	\arrow[from=1-2, to=2-2]
	\arrow[from=1-4, to=1-5]
	\arrow[from=1-5, to=2-5]
	\arrow[from=2-1, to=2-2]
	\arrow["{\Delta_I}", from=2-4, to=1-4]
	\arrow[from=2-4, to=2-5]
\end{tikzcd}\]
\[\begin{tikzcd}
	& A \\
	I && J
	\arrow[from=2-1, to=1-2]
	\arrow["{{\iota_J = \epsilon_I}}"', from=2-1, to=2-3]
	\arrow[from=2-3, to=1-2]
\end{tikzcd}\]
    Dually, one can define \textit{double comonoids}. 

    A morphism between two double (co)monoids is a morphism of the two underlying (co)monoids. 
    
    The categories of double (co)monoids in the 2-monoidal category $(\catname{C},\square, I, \circ, J)$ are denoted by $\text{dMon}(\catname{C},\square, I, \circ, J)$ and by $\text{dComon}(\catname{C},\square, I, \circ, J)$, respectively.
\end{definition}

The 2-monoidal structure of the category of symmetric sequences allows us to relate different categories of monoids, comonoids, bimonoids, and double (co)monoids to each other through categorical equivalences. We denote by $\text{Mon}(\C,\otimes,\textbf{1})$ the category of monoids in the monoidal category $(\C,\otimes,\textbf{1})$ and by $\text{Comon}(\C,\otimes,\textbf{1})$ the corresponding category of comonoids. The next result follows from Proposition 6.36 of \cite{Aguiar2010}.

\begin{corollary}[\cite{Aguiar2010}]
    There are canonical equivalences of categories $$\text{Bimon}(\catname{sSeq},\square,\circ) \cong \text{Comon}(\text{Mon}(\catname{sSeq},\square),\circ) \cong \text{Mon}(\text{Comon}(\catname{sSeq},\circ),\square),$$ $$\text{dMon}(\catname{sSeq},\square,\circ) \cong \text{Mon}(\text{Mon}(\catname{sSeq},\square),\circ),$$ $$\text{dComon}(\catname{sSeq},\square,\circ) \cong \text{Comon}(\text{Comon}(\catname{sSeq},\circ),\square).$$
\end{corollary}

\subsection{Operads are Day-operads} As a consequence of the machinery introduced in Section \ref{section_main_theorem}, we are able to prove the following lemma.

\begin{lemma} \label{Interchange_isomorphisms}
    The interchange law is compatible with isomorphisms, in the sense that the following diagram commutes.
\[\begin{tikzcd}
	{(\PP \circ \mathcal{J}) \square (\mathcal{J} \circ \PP)} && {(\PP \square \mathcal{J} ) \circ (\mathcal{J} \square \PP)} \\
	{\PP \square \PP} && {\PP \circ \PP} \\
	{(\mathcal{J} \circ \PP) \square (\PP \circ \mathcal{J})} && {(\mathcal{J} \square \PP) \circ (\PP \square \mathcal{J})}
	\arrow["{{\sigma_{\PP,\mathcal{J},\mathcal{J},\PP}}}", from=1-1, to=1-3]
	\arrow["\cong", from=1-3, to=2-3]
	\arrow["\cong", from=2-1, to=1-1]
	\arrow["\cong"', from=2-1, to=3-1]
	\arrow["{{\sigma_{\mathcal{J},\PP,\PP,\mathcal{J}}}}"', from=3-1, to=3-3]
	\arrow["\cong"', from=3-3, to=2-3]
\end{tikzcd}\]
\end{lemma}
\begin{proof}
    As before, we pass to the category of multiplicative functors and use the notation introduced in Lemma \ref{equivalence}. It therefore suffices to prove that the following diagram commutes for any $$((f,g),(f',g')) \in \catname{Cospan(\catname{Comp})},$$ where $S \xrightarrow{f}R, T \xrightarrow{g}S, S' \xrightarrow{f'}R,$ and $T' \xrightarrow{g'}S'$ are bijections. Note that if any of the latter is not a bijection, the definition of the functor $\Gamma$ makes the diagram commute trivially. Consider the diagram
\[\begin{tikzcd}[column sep=scriptsize]
	{\Phi_{\PP}(f) \times \Gamma(g) \times \Gamma(f') \times \Phi_{\PP}(g')} && {\Phi_{\PP}(f) \times \Gamma(f') \times \Gamma(g \times_R S') \times \Phi_{\PP}(S \times_R g')} \\
	{\Phi_{\PP}(fg) \times \Phi_{\PP}(f'g')} && {\Phi_{\PP}(f \times_R f') \times \Phi_{\PP}\hspace{-0.05cm}\left((g \times_R S') \times_{S \times_RS'}(S \times_R g')\right)} \\
	{\Gamma(f) \times \Phi_{\PP}(g) \times \Phi_{\PP}(f') \times \Gamma(g')} && {\Gamma(f) \times \Phi_{\PP}(f') \times \Phi_{\PP}(g \times_R S') \times \Gamma(S \times_R g')}
	\arrow["t", from=1-1, to=1-3]
	\arrow[from=1-3, to=2-3]
	\arrow[from=2-1, to=1-1]
	\arrow[from=2-1, to=3-1]
	\arrow["t"', from=3-1, to=3-3]
	\arrow[from=3-3, to=2-3]
\end{tikzcd}\]
where all morphisms are the evident ones. As the functor $\Gamma$ carries bijections to the monoidal unit (see Lemma \ref{operations_kan} for details), we can reduce the above diagram to the following one.
\begin{equation} \label{diagram_lemma_interchange} \begin{tikzcd}[column sep=scriptsize]
	{\Phi_{\PP}(f) \times \Phi_{\PP}(g')} && {\Phi_{\PP}(f) \times \Phi_{\PP}(S \times_R g')} \\
	{\Phi_{\PP}(fg) \times \Phi_{\PP}(f'g')} && {\Phi_{\PP}(f \times_R f') \times \Phi_{\PP}\hspace{-0.05cm}\left((g \times_R S') \times_{S \times_RS'}(S \times_R g')\right)} \\
	{\Phi_{\PP}(g) \times \Phi_{\PP}(f')} && {\Phi_{\PP}(f') \times \Phi_{\PP}(g \times_R S').}
	\arrow["t", from=1-1, to=1-3]
	\arrow[from=1-3, to=2-3]
	\arrow[from=2-1, to=1-1]
	\arrow[from=2-1, to=3-1]
	\arrow["t"', from=3-1, to=3-3]
	\arrow[from=3-3, to=2-3]
\end{tikzcd}\end{equation}
The two morphisms on the right are given by the maps of pairs represented in the following diagram
\[\begin{tikzcd}
	& {S \times_RS'} & {T \times_RS'} \\
	R & {S \times_RS'} & {S \times_RT'} \\
	R & {S'}
	\arrow["\id"', from=1-2, to=2-2]
	\arrow["{{{g \times_R S'}}}"', from=1-3, to=1-2]
	\arrow["\cong", from=1-3, to=2-3]
	\arrow["\id"', from=2-1, to=3-1]
	\arrow["{{{f \times_R f'}}}"', from=2-2, to=2-1]
	\arrow["{{{f \times_R S'}}}", from=2-2, to=3-2]
	\arrow["{{{S \times_R g'}}}", from=2-3, to=2-2]
	\arrow["{{{f'}}}", from=3-2, to=3-1]
\end{tikzcd}\]
where the map marked with $\cong$ is the obvious one that makes the square commute, all other maps being bijective. By the usual properties of the fiber product, the map $$(g \times_R S') \times_{S \times_RS'}(S \times_R g')$$ is isomorphic to $g \times_R g'$, and therefore Diagram (\ref{diagram_lemma_interchange}) commutes by the construction of $t$.
\end{proof}

The previous lemma shows that morphisms of the type $$\PP \square \PP \rightarrow \PP \circ \PP$$ provided by the interchange law are canonical. This will be useful in the proof of the next result that connects the notions of operad and of Day-operad. We note briefly that an analogous result holds for morphisms of the type $$\PP \square \mathcal{Q} \rightarrow \mathcal{P} \circ \mathcal{Q}$$ induced by the interchange law. The reasoning is left to the reader; it is purely technical.

\begin{corollary} \label{corollary_operad_Day}
    Every operad admits a natural Day-operad structure.
\end{corollary}
\begin{proof}
    Let $\PP$ be an operad. We precompose $\mu: \PP \circ \PP \rightarrow \PP$ with $\sigma: \PP \square \PP \rightarrow \PP \circ \PP$, where the latter may be understood in any of the two equivalent ways presented in Lemma \ref{Interchange_isomorphisms}. Hence, we obtain the Day-multiplication $\widetilde \mu: \PP \square \PP \rightarrow \PP$. Let us verify that all relevant diagrams commute using the fact that the category $(\catname{sSeq}, \square, \mathcal{J}, \circ, \mathcal{J})$ is $2$-monoidal. In fact, we need to show that the following diagrams commute. 
\[\begin{tikzcd}
	{(\PP \square \PP) \square \PP} && {\PP \square (\PP \square \PP)} \\
	{\PP \square \PP} && {\PP \square \PP} \\
	& \PP
	\arrow["{{{{a^\square_{\PP,\PP,\PP}}}}}", from=1-1, to=1-3]
	\arrow["{{{\widetilde \mu \square \id_{\PP}}}}"', from=1-1, to=2-1]
	\arrow["{{{\id_{\PP} \square \widetilde \mu}}}", from=1-3, to=2-3]
	\arrow["{\widetilde \mu}"', from=2-1, to=3-2]
	\arrow["{\widetilde \mu}", from=2-3, to=3-2]
\end{tikzcd}\]
\[\begin{tikzcd}
	{\mathcal{J} \square \PP} && {\PP \square \PP} && {\PP \square \mathcal{J}} && {\PP \square \PP} \\
	& \PP &&&& \PP
	\arrow["{{{\eta \square \id_{\PP}}}}", from=1-1, to=1-3]
	\arrow["{{{l^\square_{\PP}}}}"', from=1-1, to=2-2]
	\arrow["{\widetilde \mu}", from=1-3, to=2-2]
	\arrow["{{{\id_{\PP} \square \eta}}}", from=1-5, to=1-7]
	\arrow["{{{r^\square_{\PP}}}}"', from=1-5, to=2-6]
	\arrow["{\widetilde \mu}", from=1-7, to=2-6]
\end{tikzcd}\]
    Let us look at the next diagram.
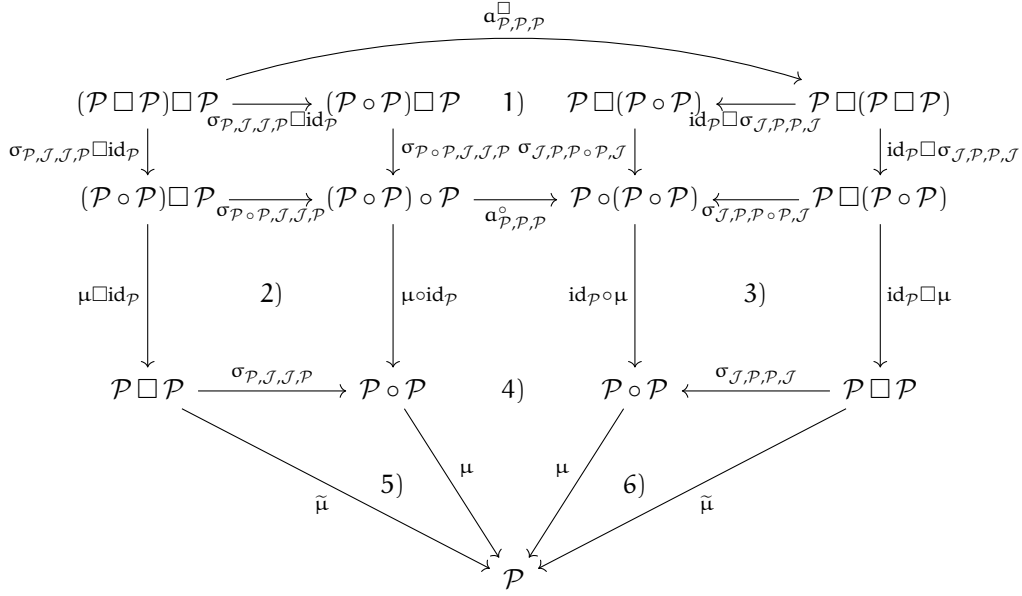
\[\begin{tikzcd}[ampersand replacement=\&, column sep=tiny]
	{(\PP \square \PP) \square \PP} \&\& {(\PP \circ \PP) \square \PP} \& {1)} \& {\PP \square (\PP \circ \PP)} \&\& {\PP \square (\PP \square \PP)} \\
	{(\PP \circ \PP) \square \PP} \&\& {(\PP \circ \PP) \circ \PP} \&\& {\PP \circ (\PP \circ \PP)} \&\& {\PP \square (\PP \circ \PP)} \\
	\& {2)} \&\&\&\& {3)} \\
	{\PP \square \PP} \&\& {\PP \circ \PP} \& {4)} \& {\PP \circ \PP} \&\& {\PP \square \PP} \\
	\&\& {5)} \&\& {6)} \\
	\&\&\& \PP
	\arrow["{{{{\sigma_{\PP,\mathcal{J},\mathcal{J},\PP} \square \id_{\PP}}}}}"', from=1-1, to=1-3]
	\arrow["{{{{a^\square_{\PP,\PP,\PP}}}}}", curve={height=-30pt}, from=1-1, to=1-7]
	\arrow["{{{{\sigma_{\PP,\mathcal{J},\mathcal{J},\PP} \square \id_{\PP}}}}}"', from=1-1, to=2-1]
	\arrow["{{{{\sigma_{\PP \circ \PP,\mathcal{J},\mathcal{J},\PP}}}}}", from=1-3, to=2-3]
	\arrow["{{{{\sigma_{\mathcal{J},\PP,\PP \circ \PP,\mathcal{J}}}}}}"', from=1-5, to=2-5]
	\arrow["{{{{\id_{\PP} \square \sigma_{\mathcal{J},\PP,\PP, \mathcal{J}}}}}}", from=1-7, to=1-5]
	\arrow["{{{{\id_{\PP} \square \sigma_{\mathcal{J},\PP,\PP,\mathcal{J}}}}}}", from=1-7, to=2-7]
	\arrow["{{{{\sigma_{\PP \circ \PP,\mathcal{J},\mathcal{J},\PP}}}}}"', from=2-1, to=2-3]
	\arrow["{{{{\mu \square \id_{\PP}}}}}"', from=2-1, to=4-1]
	\arrow["{{{{a^\circ_{\PP,\PP,\PP}}}}}"', from=2-3, to=2-5]
	\arrow["{{{{\mu \circ \id_{\PP}}}}}", from=2-3, to=4-3]
	\arrow["{{{{\id_{\PP} \circ \mu}}}}"', from=2-5, to=4-5]
	\arrow["{{{{\sigma_{\mathcal{J},\PP,\PP \circ \PP,\mathcal{J}}}}}}", from=2-7, to=2-5]
	\arrow["{{{{\id_{\PP} \square \mu}}}}", from=2-7, to=4-7]
	\arrow["{{{{\sigma_{\PP,\mathcal{J},\mathcal{J},\PP}}}}}", from=4-1, to=4-3]
	\arrow["{{{\widetilde \mu}}}"', from=4-1, to=6-4]
	\arrow["\mu", from=4-3, to=6-4]
	\arrow["\mu"', from=4-5, to=6-4]
	\arrow["{{{{\sigma_{\mathcal{J},\PP,\PP,\mathcal{J}}}}}}"', from=4-7, to=4-5]
	\arrow["{{{\widetilde \mu}}}", from=4-7, to=6-4]
\end{tikzcd}\]
    The squares labeled 2) and 3) and the pentagon labeled 4) commute by naturality of $\sigma$ and the operad axioms for $\PP$. The triangles labeled 5) and 6) commute by the definition of $\widetilde \mu$. It remains to prove commutativity of the hexagon labeled 1). Consider the following diagram, in which we have omitted all unitors and identified isomorphic symmetric sequences accordingly. The identifications can be deduced from the labels on the morphisms.
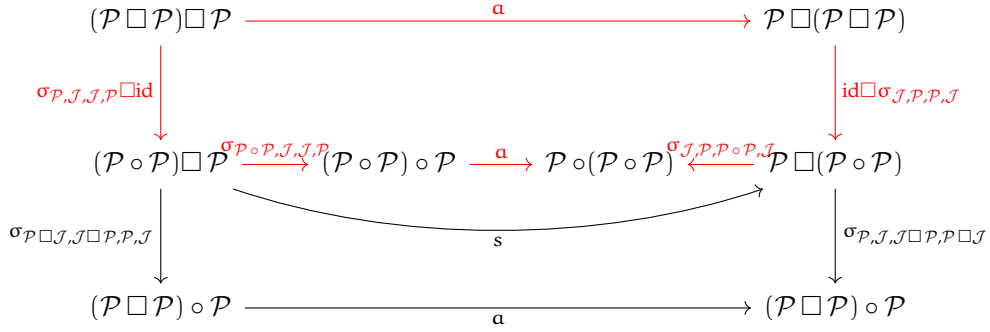
\[\begin{tikzcd}
	{(\PP \square \PP) \square \PP} &&& {\PP \square (\PP \square \PP)} \\
	\\
	{(\PP \circ \PP) \square \PP} & {(\PP \circ \PP) \circ \PP} & {\PP \circ (\PP \circ \PP)} & {\PP \square (\PP \circ \PP)} \\
	\\
	{(\PP \square \PP) \circ \PP} &&& {(\PP \square \PP)\circ \PP}
	\arrow["a", color=red, from=1-1, to=1-4]
	\arrow["{{{{{\sigma_{\PP,\mathcal{J},\mathcal{J},\PP} \square \id}}}}}"', color=red, from=1-1, to=3-1]
	\arrow["{{{{{\id \square \sigma_{\mathcal{J},\PP,\PP,\mathcal{J}}}}}}}", color=red, from=1-4, to=3-4]
	\arrow["{{\sigma_{\PP \circ \PP, \mathcal{J},\mathcal{J},\PP}}}", color=red, from=3-1, to=3-2]
	\arrow["s"', curve={height=30pt}, from=3-1, to=3-4]
	\arrow["{{{{{\sigma_{\PP \square \mathcal{J}, \mathcal{J} \square \PP, \PP,\mathcal{J}}}}}}}"', from=3-1, to=5-1]
	\arrow["a", color=red, from=3-2, to=3-3]
	\arrow["{{\sigma_{\mathcal{J},\PP,\PP \circ \PP, \mathcal{J}}}}"', color=red, from=3-4, to=3-3]
	\arrow["{{{{{\sigma_{\PP,\mathcal{J},\mathcal{J} \square \PP, \PP \square \mathcal{J}}}}}}}", from=3-4, to=5-4]
	\arrow["a"', from=5-1, to=5-4]
\end{tikzcd}\]
Here $s$ denotes the braiding of the matrix monoidal structure, which corresponds to swapping sets in a product. The commutativity of the outer rectangle follows from Diagram (\ref{diagram_2-monoidal_associativity_1}) in Definition \ref{definition_2-monoidal}, together with Theorem \ref{2-monoidal}. The inner region of the diagram commutes by Lemma \ref{Interchange_isomorphisms} and the explicit description of the braiding. Finally, the bottom part commutes by the observation regarding the generalization of Lemma \ref{Interchange_isomorphisms}. Since the subdiagram highlighted in red is precisely hexagon 1), this completes the first part of the proof.

    Now, consider the diagram
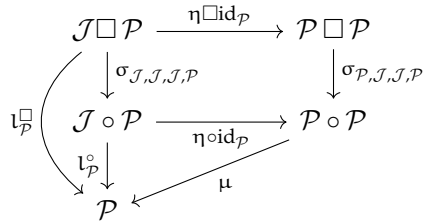
\[\begin{tikzcd}
	{\mathcal{J} \square \PP} && {\PP \square \PP} \\
	{\mathcal{J} \circ \PP} && {\PP \circ \PP} \\
	\PP
	\arrow["{{\eta \square \id_{\PP}}}", from=1-1, to=1-3]
	\arrow["{{\sigma_{\mathcal{J}, \mathcal{J}, \mathcal{J}, \PP}}}", from=1-1, to=2-1]
	\arrow["{{l^\square_{\PP}}}"', curve={height=30pt}, from=1-1, to=3-1]
	\arrow["{{\sigma_{\PP, \mathcal{J}, \mathcal{J}, \PP}}}", from=1-3, to=2-3]
	\arrow["{{\eta \circ \id_{\PP}}}"', from=2-1, to=2-3]
	\arrow["{{l^\circ_{\PP}}}"', from=2-1, to=3-1]
	\arrow["\mu", from=2-3, to=3-1]
\end{tikzcd}\]
    and note that the square commutes by naturality of $\sigma$. The commutativity of the left triangle is an immediate consequence of the unitality diagrams appearing in Definition \ref{definition_2-monoidal} and the fact that $(\catname{sSeq}, \square, \mathcal{J}, \circ, \mathcal{J})$ is $2$-monoidal. The second triangle commutes because $\PP$ is an operad. This shows that the outer diagram commutes as well.

    Finally, we examine the diagram
\[\begin{tikzcd}
	{\PP \square \mathcal{J}} && {\PP \square \PP} \\
	{\PP \circ \mathcal{J}} && {\PP \circ \PP} \\
	\PP
	\arrow["{ \id_{\PP} \square \eta}", from=1-1, to=1-3]
	\arrow["{\sigma_{\PP, \mathcal{J}, \mathcal{J}, \mathcal{J}}}", from=1-1, to=2-1]
	\arrow["{r^\square_{\PP}}"', curve={height=30pt}, from=1-1, to=3-1]
	\arrow["{\sigma_{\PP, \mathcal{J}, \mathcal{J}, \PP}}", from=1-3, to=2-3]
	\arrow["{\id_{\PP} \circ \eta}"', from=2-1, to=2-3]
	\arrow["{r^{\circ}_{\PP}}"', from=2-1, to=3-1]
	\arrow["\mu", from=2-3, to=3-1]
\end{tikzcd}\]
    to get the last commutativity, we need to conclude that the multiplication $\widetilde \mu$ endows $\PP$ with the structure of a Day-operad.
\end{proof}

In \cite{Balteanu_iterated}, the authors introduce the notion of a 2-fold monoidal category. In the special case where the two monoidal units coincide, this definition aligns with that of a 2-monoidal category as formulated in \cite{Aguiar2010} (see Definition \ref{definition_2-monoidal}). As a result, the framework developed in \cite{Balteanu_iterated} applies to the category of symmetric sequences. In particular, we obtain the following consequence of Theorem \ref{2-monoidal}.

\begin{corollary}[\cite{Balteanu_iterated}] \label{corollary_2-fold_loop}
The group completion of the nerve of $(\catname{sSeq}, \square, \mathcal{J}, \circ, \mathcal{J})$ is a $2$-fold loop space.
\end{corollary}

\bibliographystyle{alpha}
\bibliography{biblio}

\end{document}